\documentclass[10pt]{article}
\usepackage{amssymb,amsmath,latexsym,color}
\usepackage{amscd,amsthm}
\usepackage{titlesec}

\titleformat*{\subsection}{\normalsize\itshape}

\newtheorem{theorem}[equation]{Theorem}
\newtheorem{cor}[equation]{Corollary}
\newtheorem{lemma}[equation]{Lemma}
\newtheorem{proposition}[equation]{Proposition}

\newtheorem{thm}{Theorem}

\newtheorem*{citethmdominant}{Theorem \ref{dominant}}
\newtheorem*{citethminjectivity}{Theorem \ref{injectivity}}
\newtheorem*{citethmfinitedim}{Theorem \ref{finitedim}}
\newtheorem*{citethmfinitegen}{Theorem \ref{finitegen}}
\newtheorem*{citethmfinitepres}{Theorem \ref{finitepres}}
\newtheorem*{citethmFPhypothesis}{Theorem \ref{FPhypothesis}}
\newtheorem*{citethmGHWhypothesis}{Theorem \ref{GHWhypothesis}}

\newtheorem{Zthm}[equation]{Zelmanov Nil Theorem}
\newtheorem{Zcor}[equation]{Zelmanov Nil Corollary}
\newtheorem{Zlemma}[equation]{Zassenhaus Lemma}
\newtheorem{ZClassthm}[equation]{Zelmanov Classification Theorem}

\numberwithin{equation}{section}

\newcommand{\A}{{\mathbb A}}
\newcommand{\F}{\mathbb{F}}
\newcommand{\Q}{{\mathbb Q}}
\newcommand{\K}{\mathbb{K}}

\newcommand{\Z}{\mathbb{Z}}

\newcommand{\ot}{\otimes}

\newcommand{\Min}{\mathrm{Min}}

\newcommand{\gs}{\sigma}

\newcommand{\End}{\hbox{End}}

\newcommand{\B}{{\rm B}}

\newcommand{\SL}{{\rm SL}}
\newcommand{\Hom}{{\rm Hom}}
\newcommand{\id}{{\rm id}}

\newcommand{\Ker}{{\rm Ker}}

\newcommand\Inn{\text{\rm Inn}}

\newcommand{\bSL}{{\mathbf{SL}}}
\newcommand{\bPSL}{{\mathbf{PSL}}}

\newcommand{\lto}{\leftarrow}

\newcommand{\ad}{\mathrm{ad}}

\newcommand{\Coker}{\mathrm{Coker}}
\renewcommand{\d}{\mathrm{d}}
\newcommand{\Der}{\mathrm{Der}}
\newcommand{\Ext}{\mathrm{Ext}}
\newcommand{\Ht}{\mathrm{ht}}
\newcommand{\Image}{\mathrm{Im}}
\newcommand{\Ind}{\mathrm{Ind}}

\newcommand{\Mod}{\mathrm{Mod}}

\newcommand{\sgn}{\mathrm{sgn}}
\newcommand{\Soc}{\mathrm{Soc}}

\newcommand{\Tr}{\mathrm{Tr}}

\newcommand{\fb}{\mathfrak{b}}
\newcommand{\fB}{\mathfrak{B}}

\newcommand{\fg}{\mathfrak{g}}
\newcommand{\fG}{\mathfrak{G}}
\newcommand{\fh}{\mathfrak h}

\newcommand{\fP}{\mathfrak{P}}
\newcommand{\fs}{\mathfrak{s}}

\newcommand{\fu}{\mathfrak{u}}
\newcommand{\fz}{\mathfrak{z}}

\newcommand{\fsl}{\mathfrak{sl}}

\title{ Free Jordan Algebras and\\ Representations of
$\widehat\fsl_2(J)$}

\author{Michael Lau$\hbox{}^{1*}$\ \  and \ Olivier Mathieu$\hbox{}^{2}$\thanks{The authors gratefully acknowledge funding from the Natural Sciences and Engineering Research Council of Canada (M.L.) and UMR 5208 of the Centre national de la recherche scientifique (O.M.).}
\vspace{0.3cm} 
\\
$\hbox{\ \,}^1${\small Universit\'e Laval}\\ {\small D\'epartement de math\'ematiques et de statistique}\\{\small Qu\'ebec, QC,
Canada G1V 0A6}\\ {\small Michael.Lau@mat.ulaval.ca}\\
\\
$\hbox{}^{2 }${\small Institut Camille Jordan}\\ {\small UMR 5028 du CNRS, Universit\'e Claude Bernard Lyon 1}\\
{\small 69622 Villeurbanne Cedex, France}\\ {\small Shenzhen International Center for Mathematics, Shenzhen, China}\\{\small mathieu@math.univ-lyon1.fr}\vspace{0.1cm}}

\date{}

\begin{document}
\maketitle
\setcounter{section}{0}

\begin{small}
  \noindent
      {\bf Abstract:}   Let $J$ be a unital Jordan algebra, and
      let $\widehat\fsl_2(J)$ be the universal central extension of the Tits-Kantor-Koecher  Lie algebra. 

In Part A, we study the category of 
$(\widehat\fsl_2(J), \SL_2(\K))$-modules. We characterize the dominant $J$-spaces, which are analogous to the dominant highest weights appearing in classical settings.  A family of universal envelopes $\mathcal{U}_n(J)$ associated to such modules is introduced and studied.  We also prove some finiteness theorems.

In Part C, we define the notion of smooth $\widehat\fsl_2(J)$-modules for augmented Jordan algebras $J$, and investigate the category of smooth modules in the spirit of Cline-Parshall-Scott highest weight categories. We
 show that the standard modules of this category are finite dimensional when $J$ is finitely generated.

The free unital Jordan algebra $J(D)$ over $D$ variables
is an elusive object, but finiteness and $\hbox{Ext}$-vanishing properties suggest that the smooth $\widehat\fsl_2(J(D))$-modules with even eigenvalues might form a generalized highest weight category.  However, we prove that such an assertion would contradict recently obtained information about the growth of free Jordan algebras.  See \cite{KM} and \cite{DH} for more details.  It then follows that the category of smooth $\widehat\fsl_2(J(D))$-modules with even eigenvalues is not a generalized highest weight category when $D\geq 2$.

Surprisingly, the proofs of most of these results make use of deep theorems of E. Zelmanov.

\bigskip


\bigskip


\bigskip

\noindent
{\bf Statements and Declarations:} The authors report that they have no competing interests to declare. Data sharing is not applicable to this article as no datasets were generated or analyzed during the current study.

\end{small}
\maketitle

\section{Introduction} 

Let $\K$ be a field of characteristic zero.  
For any unital Jordan algebra $J$,  we denote by 
$\widehat{\fsl}_2(J)$ the universal central extension of the Tits-Kantor-Koecher  Lie algebra $TKK(J)$.
In the present paper, we investigate two categories 
of representations of the Lie algebra 
$\widehat{\fsl}_2(J)$.

An {\it $(\widehat\fsl_2(J),\bSL_2(\K))$-module}
is an $\widehat\fsl_2(J)$-module that is locally finite dimensional as an $\fsl_2(\K)$-module.
In Part A, we investigate finiteness properties of the
category of $(\widehat{\fsl}_2(J),\bSL_2(\K))$-modules for
an arbitrary unital Jordan algebra $J$.
Relative to the action of the usual diagonal element $h$ of $\fsl_2(\K)$, any $(\widehat\fsl_2(J),\bSL_2(\K))$-module $X$ admits an eigenspace decomposition $X=\bigoplus_{n\in\Z} X_n$.
In particular, the Lie algebra
$\fG:=\widehat\fsl_2(J)$ has a short grading $\fG=\fG_{-2}\oplus \fG_0\oplus \fG_{2}$. For integers $n$, we ask

\medskip

\centerline{\it Which $\fG_0$-modules $V$ are isomorphic to
$H^0(\fG_{2},X)\cap X_n$ for some $X$?}

\medskip

A vector space $V$ endowed with a linear map 
 $\rho:J\to\End_\K(V)$ is called a 

\noindent {\it $J$-space}  when
$V$ is a $\fG_0$-module and $\rho(a)$ is the action  of $h\otimes a\in\fG_0$ on  $V$. This is equivalent to
the identities
$[\rho(a),\rho(a^2)]=0$ and 
$[[\rho(a),\rho(b)],\rho(c)]=4\rho(a(cb)-(ac)b)$
for all $a,b,c\in J$.   A $J$-space is called {\it dominant} of level $n$ if it is isomorphic to $H^0(\fG_{2},X)\cap X_n$ for some $(\widehat\fsl_2(J),\bSL_2(\K))$-module $X$. 

Our first  result characterizes 
dominant $J$-spaces of level $n$. For any
partition $\gs=(\gs_1,\gs_2,\ldots ,\gs_m)$ of $n+1$, we write $|C_\sigma|$ for the cardinality of the corresponding conjugacy class in the symmetric group $S_{n+1}$ and $\sgn(\gs)$
for its  signature.  For $a\in J$, we write $\rho_\sigma(a)$ for the expression $\rho(a^{\gs_1})\rho(a^{\gs_2})\cdots\rho(a^{\gs_m})$.

\begin{citethmdominant}
Let $(V,\rho)$ be a $J$-space such that $\rho(1)=n$.  Then 
$V$ is dominant of level $n$ if and only if
it satisfies the following   condition:

$$\sum_{\gs\ \vdash\ n+1} \hbox{sgn}(\gs)\,|C_\gs|\,\rho_\gs(a)=0\hbox{\ for all\ } a\in J.$$
\end{citethmdominant}

There is thus  an associative algebra 
${\mathcal U}_n(J)$, endowed with a linear map $\rho_n:J\to {\mathcal U}_n(J)$, which is the unital universal associative envelope of $J$ in the category of dominant $J$-spaces of level $n$. 
We prove the following result:

\begin{citethminjectivity} 
The map $\rho_n$ is injective for any $n\geq 4$.
\end{citethminjectivity}

Next we prove two finiteness theorems:

\begin{citethmfinitegen} The associative algebra ${\mathcal U}_n(J)$ is finitely generated, for all finitely generated Jordan algebras $J$.
\end{citethmfinitegen}

\begin{citethmfinitepres} The Lie algebra $\widehat\fsl_2(J)$ is finitely presented, for all finitely presented Jordan algebras $J$.
\end{citethmfinitepres}

The free Jordan $\K$-algebra $J(D)$  on $D$ generators admits a grading $J(D)=\bigoplus_{n\geq 0} J_n(D)$, where $J_0(D)=\K$ and $J_1(D)$ is the linear span of the $D$ generators.  The algebra $J(D)$ is an elusive object.  However, an explicit formula for $\dim J_n(D)$ has been conjectured in \cite{KM}. The previous theorem shows that $\widehat\fsl_2(J(D))$ is
$FP_2$.  This theorem and other observations raise the natural question of whether the
Lie algebra $\widehat\fsl_2(J(D))$ could be $FP_\infty$, that is, whether its trivial module has a projective resolution consisting of finitely generated modules.  Such a finiteness condition would have profound consequences for the growth of free Jordan algebras:

\begin{citethmFPhypothesis} Assume that
$\widehat\fsl_2(J(D+1))$ is of type $FP_\infty$.
Then the conjectural formula of \cite{KM} for 
$\dim J_n(D)$ holds.
\end{citethmFPhypothesis}



Part $B$ contains the categorical background (used in Part C) on good filtrations in locally artinian categories.  In particular, we introduce the notion of {\em generalized highest weight categories} ({\em generalized HW categories} for short), which differ from highest weight categories in that they allow nontrivial self-extensions of standard objects.  We close Part B with an Ext-vanishing result for generalized HW categories (Theorem \ref{generalized HWvanishing}).

\medskip
It has been conjectured in \cite{KM} that
$H_k(\widehat\fsl_2(J(D),\fsl_2(\K))=0$ for $k>0$. By
\cite{H}, the dual of this homology group describes the
self-extensions of $\K$ in the large category of
$(\widehat\fsl_2(J),\bSL_2(\K))$-modules. 
In Part C, we introduce the smaller and more manageable category of smooth modules, with a similar connection between 
$H_*(\widehat\fsl_2(J(D),\fsl_2(\K))$ and self-extensions of the trivial module $\K$.

A unital Jordan algebra $J$ endowed with an ideal $J_+$ of codimension $1$ is called an
{\it augmented} Jordan algebra. An 
$\widehat\fsl_2(J)$-module is called {\it smooth}
if any cyclic submodule is a finite dimensional 
$\widehat\fsl_2(J/J_+^n)$-module  for some $n$.  This terminology comes from proalgebraic groups, where smoothness means that a module is locally a representation for some algebraic quotient.  For $p$-adic groups, the definition of smoothness is similar.

In Part $C$, we investigate the category of smooth modules
in the spirit of the Cline-Parshall-Scott paper \cite{cps88}.
 For an augmented Jordan algebra $J$ and $n\geq 0$, there is a  $1$-dimensional $J$-space $\K_n$ defined by $\rho(1)=n$ and $\rho(J_+)=0$. The standard module $\nabla(n)$ is a  smooth module such that
 $H^0(\fG_2, \nabla(n))=\K_n$, and  satisfying a universal property.   When $J$ is finitely generated we prove that:

  \begin{citethmfinitedim} The standard module
 $\nabla(n)$ is finite dimensional.
 \end{citethmfinitedim}

However, the category of smooth modules is almost never
 a highest weight category.  Some results suggest that the category of smooth 
$\widehat\fsl_2(J)$-modules with even eigenvalues
 could be a generalized HW category when
$J$ is a finitely generated free Jordan algebra. We conclude
the paper with the following result.

\begin{citethmGHWhypothesis} Assume that the category
of smooth $\widehat\fsl_2(J(D))$-modules with even eigenvalues
is a generalized HW category.  Then the conjectural formula of \cite{KM} for $\dim\,J_n(D)$ is true.
\end{citethmGHWhypothesis}

\medskip 
\noindent{\it About the proofs:}
The proof of Theorem 1 is elementary. For Theorem 2,
the proof of the injectivity of $\rho_n$ for even $n$ is  elementary. Strangely, it uses Zelmanov's Classification Theorem \cite{Zelmanov80,Zelmanov83} for the odd integers $n\geq 5$.

The proofs of Theorems \ref{finitegen}, 4, and 7, as well as 
Lemma \ref{finitehomology} are based on Zelmanov's
Theorem for nil Jordan algebras. Theorems 5 and 8 are 
consequences of Theorem 2 and Corollary  1 of
\cite{KM}.

\smallskip
\noindent
    {\it Recent developments} In a very recent arXiv preprint \cite{DH} posted when the present paper was under final revisions, Dotsenko and Hentzel provide detailed computational data indicating that the conjectural graded dimension formula \cite{KM} for the free Jordan algebra on two generators fails in degree 19.  When Theorems 5 and 8 are combined with Dotsenko and Hentzel's result, we obtain the following immediate and perhaps surprising consequences for Lie algebra representation theory:


\medskip

\noindent
{\bf Corollary.\ }{\it
  For $D\geq 3$, the Lie algebra $\widehat\fsl_2(J(D))$ is {\em not} of type $FP_\infty$.}

\medskip

\noindent
{\bf Corollary.\ }{\it
  For $D\geq 2$, the category of smooth $\widehat{\mathfrak{sl}}_2(J(D))$-modules with even eigenvalues is {\em not} a generalized HW category.}

\medskip
  

\noindent
As illustrated by these corollaries, the results of the present paper show definite connections between the structure of free Jordan algebras, the representation theory of the Lie algebras $\widehat\fsl_2(J(D))$, and Zelmanov's Theorem on nil Jordan algebras of bounded index.

\smallskip
\noindent
{\it Acknowledgements} The authors would like to thank
the referees for useful comments and for calling their attention to the very interesting paper of Chari, Fourier, and Khandai \cite{Chari-Fourier-Khandai}.  O.M. also thanks E. Neher and I. Shestakov
for  fruitful discussions.

\vskip1cm
\centerline{\bf \Large PART A: The category of  
$(\widehat{\fsl}_2(J),\SL_2)$-modules}
\vskip1cm

\section{The Lie algebra $\widehat{\fsl}_2(J)$}

\subsection{Conventions and notations}\label{notations}

As before, let $\K$ be a field of characteristic zero.  All vector spaces, algebras, and tensor products will be taken over $\K$, unless explicitly otherwise mentioned.  

We start with the definitions involving the Lie algebra
$\fsl_2:=\fsl_2(\K)$.
We fix the usual basis $\{h,e,f\}$ satisfying $[h,e]=2e$, $[h,f]=-2f$, and $[e,f]=h$.  The Killing form 
$\kappa$ will be normalized so that $\kappa(h,h)=4$.

For any $\fsl_2$-module $M$ and $z\in \K$, we say that $M_z:=\{m\in M\,|\,h.m=z m\}$ is an {\em eigenspace} relative to the action of $h$.  When $M_z\neq 0$, we say that $z$ is an {\em eigenvalue} of the module $M$.  All $\fsl_2$-modules considered in this paper will decompose as direct sums of eigenspaces of integer eigenvalues.  For any nonnegative integer $n$, let $L(n)$ be the simple $\fsl_2$-module with highest eigenvalue $n$.  Clearly, $L(0)$ is the trivial module $\K$, $L(1)=\K^2$ is the natural $2$-dimensional representation, and $L(n)$ is isomorphic to the space $S^nL(1)$ of symmetric tensors for any $n\geq 0$.  We will write $M(n)$ for the Verma module with highest eigenvalue $n\in\mathbb{Z}$.

We continue with the notation involving Jordan algebras.
We assume that all Jordan algebras $J$ are unital unless explicitly stated otherwise.  For every $a,b\in J$, let $\partial_{a,b}:\ J\rightarrow J$ be the linear map $\partial_{a,b}(x)=a(bx)-(ax)b$.
The Jordan identity implies that $\partial_{a,b}$ is a derivation of $J$. An {\it inner derivation} is a linear combination of the derivations $\partial_{a,b}$. The vector space $\Inn\,J$ of inner derivations is a Lie algebra.

For $x\in\fsl_2$ and $a\in J$, the tensor product $x\ot a\in\fsl_2\ot J$ will be denoted $x(a)$, though $x(1)=x\ot 1$ will also be denoted by $x$.

\subsection{Tits-Kantor-Koecher construction}

In his 1962 paper \cite{T}, Tits defined a Lie algebra structure on the space 
$$TKK(J):=
(\fsl_2\otimes J)\oplus \Inn\,J,$$
with Lie bracket 
\begin{enumerate} 
\item[\rm(T1)]
$[x(a),y(b)]=[x,y](ab)+ \kappa(x,y)\partial_{a,b}$

 \item[\rm(T2)] $[\partial, x(a)]=x(\partial\,a)$,

 \item[\rm(T3)] $[\partial,\partial^\prime]=\partial\circ\partial^\prime-\partial^\prime\circ\partial,$
 \end{enumerate}

\noindent for all $x,y\in\fsl_2$, $\partial,\partial^\prime\in\Inn\,J$, and $a,b\in J$.  This construction was later generalized to Jordan pairs and triple systems by Kantor and Koecher, and $TKK(J)$ is commonly known as the {\em Tits-Kantor-Koecher (or TKK) algebra}.  It should be noted that the Tits construction, associating a Lie algebra $TKK(J)$ to each Jordan algebra $J$, is not functorial.

\subsection{Tits-Allison-Gao construction}

The Lie algebra $TKK(J)$ is perfect and its
universal central extension, denoted by
$\widehat{\fsl_2}(J)$, was nicely described by Allison and Gao \cite{AG}.  See also \cite{ABG}, \cite{neher}.

As a vector space,

\centerline{$\widehat{\fsl}_2(J)=\left({\fsl}_2\otimes J\right)\oplus\{J,J\}$,}

\noindent where $\{J,J\}=(\bigwedge^2 J)/\mathcal{S}$ and
$\mathcal{S}=\hbox{Span}\{a\wedge a^2\mid \ a
\in J\}$.
For any $a,b\in J$, we write 
$\{a,b\}$ for the image of  $a\wedge b$ in $\{J,J\}$.
The bracket on $\widehat{\fsl}_2(J)$ is given by

\begin{enumerate}
\item[{\rm (R1)}]$[x(a),y(b)]=[x,y](ab)+
\kappa(x,y)\{a,b\}$

\item[{\rm (R2)}]$[\{a,b\},x(c)]=x(\partial_{a,b}\,c)$

\item[{\rm (R3)}]$[\{a,b\},\{c,d\}]=\{\partial_{a,b}\,c,d\}+\{c,\partial_{a,b}\,d \}$.
\end{enumerate}

\noindent for all $a,b,c,d\in J$ and $x,y\in\mathfrak{sl}_2$. 
It is a bit tricky to show that 
(R3) is skew-symmetric \cite{AG}.

 Relation (R3) is similar to Relation (T3) for Tits-Kantor-Koecher algebras, since
$$\partial_{a,b}\circ\partial_{c,d}-\partial_{c,d}\circ\partial_{a,b}=\partial_{\partial_{a,b}c,d}+\partial_{c,\partial_{a,b}d}.$$
  Hence there is a Lie algebra epimorphism
$\widehat{\fsl}_2(J)\to TKK(J)$ which is the identity on
$\fsl_2\otimes J$ and sends the symbol $\{a,b\}$ to $\partial_{a,b}$.  
 It is clear from its definition that the {\em Tits-Allison-Gao construction} of $\widehat{\fsl}_2(J)$ defines a functor from Jordan algebras to Lie algebras, as is explained in \cite{KM}.  Therefore, for any ideal  $I$ in $J$,  the kernel of the Lie algebra morphism

\centerline{$\widehat{\fsl}_2(J)\to \widehat{\fsl_2}(J/I)$}

\noindent will be denoted as $\widehat{\fsl}_2(I)$.

When the Jordan algebra $J$ is associative, $TKK(J)$ is obviously $\fsl_2(J):=\fsl_2\otimes J$.  In that case, $\{J,J\}$ is the quotient $\Omega^1_J/\d J$ of K\"{a}hler $1$-forms modulo exact forms, recovering an earlier result of Bloch \cite{B}, who described the universal central extension of $\fsl_2(A)$ for commutative associative unital algebras $A$.  Specializing to the case where 
$A=\K[a]$ is a polynomial ring in one variable, we have $\widehat{\fsl}_2(\K[a])=\fsl_2(\K[a])$.  In particular, any element $a$ of a Jordan algebra $J$ induces a Lie algebra homorphism $\fsl_2(\K[a])\to \widehat{\fsl}_2(J)$.

\subsection{$(\widehat{\fsl}_2(J),\SL_2)$-modules}

Set $\SL_2=\bSL_2(\K)$.
An $\fsl_2$-module $M$ derives from a rational $\SL_2$-module if and only if $M$ is {\em locally finite} over $\fsl_2$, that is, $U(\fsl_2)v$ is finite dimensional for all $v\in M$.
This condition is clearly equivalent to the requirement that
$M$ is a direct sum of finite dimensional simple
$\fsl_2$-modules.
By definition, an  $\widehat{\fsl}_2(J)$-module $M$ is an 
{\it $(\widehat{\fsl}_2(J),\SL_2)$-module} if its structure as a module over $\fsl_2$ derives from a rational 
$\SL_2$-module.   Such modules $M$ admit an {\it eigenspace decomposition} $M=\bigoplus_{k\in \Z} M_k$, with $M_k=\{m\in M\,|\,hm=km\}$.  The eigenspace decomposition for the adjoint module
$\mathfrak{G}:=\widehat{\fsl}_2(J)$ is a short grading
\cite{neher}:
$\mathfrak{G}=\mathfrak{G}_{-2}\oplus\mathfrak{G}_{0}\oplus\mathfrak{G}_{2},$ where
\begin{eqnarray*}
\mathfrak{G}_{-2}&=&f\otimes J,\\
\mathfrak{G}_{0}&=&(h\otimes J)\oplus \{J,J\},\\
\mathfrak{G}_{2}&=&e\otimes J.
\end{eqnarray*}

\noindent
{\em For the rest of the paper, we  fix the notation $\mathfrak{G}=\widehat{\fsl}_2(J)$, so that $(\widehat{\fsl}_2(J),\SL_2)$-modules will simply be called $(\fG,\SL_2)$-modules.}

\subsection{ $J$-spaces}

In the Lie algebra $\fG_0$, the subspace $\{J,J\}$ is intricate.
However, the  $\fG_0$-modules can be conveniently described 
using the  notion of a $J$-space. 

A vector space $M$ endowed with a linear map $\rho:\ J\rightarrow\End(M)$ is called
a {\it $J$-space} if

\begin{enumerate}
\item[{\rm (J1)}] $[\rho(a),\rho(a^2)]=0$,
\item[{\rm (J2)}] $[[\rho(a),\rho(b)],\rho(c)]=  4\rho(\partial_{a,b}\,c)$,
\end{enumerate}

\noindent for all $a,b,c\in J$. 

It follows easily
from the presentation of $\fG_0$ that 
the $\fG_0$-module structures and   the $J$-space structures on a
vector space $M$ are precisely equivalent. Indeed,

\begin{enumerate}

\item[(a)]
A $\fG_0$-module structure induces a $J$-space structure
for which $\rho(a)$ is the action of $h(a)$. 

\item[(b)]
Conversely,
a $J$-space structure extends to a $\fG_0$-module structure, where the action of $h(a)$ is $\rho(a)$ and the action of
$\{a,b\}$ is $\frac{1}{4}[\rho(a),\rho(b)]$.

\end{enumerate}
  In the next section, we introduce the class of {\em dominant $J$-spaces} and give a number of  examples in this context. We will also compare them with the classical notion of Jordan bimodules. A $J$-space is said to be of {\em level n} if $\rho(1)=n\,\id$.

\section{Dominant $J$-spaces}

\noindent
A $J$-space isomorphic to a $J$-subspace of the set 
$H^0(\fG_2,X)$ of $\fG_2$-invariants for
some $(\fG,\SL_2)$-module $X$ is called {\it dominant}. 
In this section, we investigate the following question:

\medskip

\centerline{\it Which $J$-spaces  are dominant?}

\medskip

\noindent
In the trivial case where $J=\K$, the Lie algebra 
$\fG_0$ is the Cartan subalgebra $\K h$ of 
$\widehat{\fsl}_2(J)=\fsl_2$, so a $J$-space $(V,\rho)$ is dominant precisely when $\rho(1)$ is diagonalizable with nonnegative integer eigenvalues.

\subsection{Joseph functors and Weyl modules}

We now reformulate our question in terms of Weyl modules.

Set $\fB=\fG_0\oplus\fG_2$,  $\fb=\K h\oplus \K e$, and let $\B\subset\SL_2$ be the Borel subgroup with Lie algebra $\fb$. Recall that a $\fb$-module $V$ derives from a $\B$-module if $h\vert_V$ is diagonalizable with integer eigenvalues and $e\vert_V$ is locally nilpotent.
By definition, a $(\fB,\B)$-module is a 
$\fB$-module whose infinitesimal $\fb$-action derives
from a $B$-module structure.

Let $V$ be a $(\fB,B)$-module.
The induced module

\centerline{$M(V):=\Ind_\fB^\fG\,V,$}

\noindent is called the {\em generalized Verma module}
generated by $V$. Since $\ad^3(e)=0$, the subspace $Z(V):=\bigcap_{k\geq 0}\,e^k\,M(V)$ is a
$\fG$-submodule. The $\fG$-module

\centerline{$\Delta(V):=M(V)/Z(V)$}

\noindent is called the {\it Weyl module}
generated by $V$.
It satisfies the universal property
stated in the next lemma.

\begin{lemma}\label{universal} For any  $(\fB,B)$-module $V$, the Weyl module $\Delta(V)$ is a 
$(\fG,\SL_2)$-module.  Moreover, for any $(\fG,\SL_2)$-module $X$,
every morphism of $\fB$-modules $\phi:V\to X$
uniquely extends to a morphism of $\fG$-modules 
$\psi:\Delta(V)\to X$.
\end{lemma}

\begin{proof} The first assertion follows from the fact
that an $\fsl_2$-module $X$ derives from an
$\SL_2$-module  if and only if

\begin{enumerate}
\item[(a)]
$X$ is locally finite dimensional
as a $\fb$-module, and 

\item[(b)] $\bigcap_{k\geq 0}\,e^kX=0$.
\end{enumerate}

\noindent It follows that $\Delta(V)$ is a 
$(\fG,\SL_2)$-module. Moreover,
every morphism of $\fB$-modules $\phi:V\to X$
extends to a morphism of $\fG$-modules 
$\psi:M(V)\to X$ that factors through $\Delta(V)$.
\end{proof}

Any $J$-space $V$ of nonnegative integer level $n$ can be seen as a $(\fB,\B)$-module with a trivial action of
$\fG_2$. Our question thus becomes:

\medskip

\centerline{\it Which $J$-spaces $V$ embed in the Weyl module $\Delta(V)$?}

\medskip

\noindent The answer  will require  two combinatorial formulas, to be stated and proved in the next two subsections.

\bigskip
\noindent{\bf Remark:}  The notion of Weyl modules dates back to
the 1980's in the work of S. Donkin  \cite{D82}
and W. Koppinen \cite{Ko84}. For $J=\K[a]$,  the functor $V\mapsto \Delta(V)$ is the Joseph functor \cite{Joseph} for the affine
Lie algebra $A_1^{(1)}$, see \cite[ch.III]{M88} and 
\cite{CP01}.


\subsection{Garland's formula}\label{garland-subsection}

If $J=\K[a]$ is the polynomial ring in one variable, then $\fG=\fsl_2(\K[a])$.  Let $\pi:\ U(\fG)\rightarrow U(\fG_{-2}\oplus\fG_0)$ be the projection relative to the decomposition

\centerline{$U(\fG)=U(\fG_{-2}\oplus\fG_0)\oplus U(\fG)\fG_2$.}

\noindent
In what follows, we will compute $\pi(e^{(m)} f(a)^{(n)})$, where 
we write $x^{(m)}$ for the divided power $x^m/m!$.  In fact, the computation will be done in a much larger
algebra than $U(\fG)$.

 Set

\centerline{
${\overline U}(\fG)=\varprojlim\,U(\fsl_2(\K[a]/
(a^n))$,}

\noindent  and let $t$ be a formal variable. 
The following formula, due to Garland \cite{Garland78} and reformulated in \cite{CP01}, involves formal series in the large algebra ${\overline U}(\fG)[[t]]$.  As neither reference included a proof, we prove the formula here.

\begin{lemma} \label{garland} Let $n\geq m\geq 0$ be
  integers.  Then $\pi(e^{(m)} f(a)^{(n)})$ is the coefficient of $t^{m}$ in the formal series
  
\centerline{$(-1)^m\left
(\sum_{r=0}^\infty f(a^{r+1})t^r\right)^{(n-m)}\exp\left(-\sum_{s=1}^\infty \frac{h(a^s)}{s}t^s\right).$}
\end{lemma}

\noindent
\begin{proof} Set

\centerline{$\Gamma:=\{G\in\SL_2(\K[[a,t]])\mid G(0,0)=\id\}$, and}

\centerline{$\fP:=\{g\in\fsl_2(\K[[a,t]])\mid g(0,0)=0\}$.}

\noindent Since the exponential map $\exp:\fP\to {\overline U}(\fG)[[t]]$
is well defined, $\Gamma$ is a subgroup of the group
of invertible elements of ${\overline U}(\fG)[[t]]$.

 An elementary matrix computation   shows that

\begin{equation}\label{eq1}
\left(^{1\,\hskip1mm-t}_{0\hskip3mm 1}\right)\hskip1mm
\left(^{1\hskip1mm0}_{a\hskip1mm 1}\right)
= \left(^{\hskip4mm1\hskip7mm 0}_{a/(1-ta)\hskip1mm1}\right)\hskip1mm
\left(^{1-ta\,\hskip7mm 0}_{\hskip2mm 0\hskip6mm 1/(1-ta)}\right)\hskip1mm\left(^{1\hskip1mm \phi(a,t)}_{0\hskip 3mm 1}\right),
\end{equation}

\noindent for some series 
$\phi(a,t)=\sum_{k\geq 0} t^k\phi_k(a)$ in $\K[[a,t]]$ with 
$\phi(0,0)=0$.

\noindent Recall that $e,h,$ and $f$ are 
the matrices 

$$e=\left(^{0\,1}_{0\,0}\right), 
h=\left(^{1\hskip2.5mm 0}_{0\,-1}\right) \hskip1mm\hbox{and}
 f=\left(^{0\,0}_{1\,0}\right),$$

\noindent so Equation (\ref{eq1}) can be written 
as the following identity in ${\overline U}(\fG)[[t]]$:

\begin{equation}\label{csillag}\exp(-te)\,\exp f(a) 
=\exp f(a/(1-ta)) \exp\,h(\log(1-ta)) \exp\big(\sum_{k\geq 0} t^k e(\phi_k(a))\big).
\end{equation}

We  extend $\pi$ to a continous $\K[[t]]$-linear map

\centerline{$\overline{\pi}:\overline{U}(\overline{\fG})[[t]]\rightarrow \overline{U}
({\fG}_{-2}\oplus{\fG}_0)[[t]],$}

\noindent where
${\overline U}({\fG}_{-2}\oplus {\fG}_{0})$
is the closure of $U(\fG_{-2}\oplus\fG_0)$ in
${\overline U}(\fG)$.
Thus  Relation (\ref{csillag}) implies the identity

$$\overline{\pi}(\exp(-te)\,\exp f(a))
=\exp f(a/(1-ta))\,  \exp\,h(\log(1-ta)).$$

 We decompose this expression according to the degree in $t$
and the eigenvalue for $\ad(h)$. On the left side,
the component of degree $m$ and eigenvalue $2(m-n)$ is

\centerline{$(-1)^m e^{(m)} f(a)^{(n)}$.}

\noindent We now look at the right side. We have

\centerline{$f(a/(1-ta))=\sum_{r=0}^\infty\,f(a^{r+1}) t^r$.}

\noindent Therefore the component of eigenvalue $2(m-n)$ on the right side is

$$\left(\sum_{r=0}^\infty\,f(a^{r+1}) t^r\right)^{(n-m)} \exp (-\sum_{s=1}^\infty \frac{h(a^s)}{s}t^s),$$

which proves Garland's formula.
\end{proof}

\subsection{The Girard-Newton formula}

We now introduce some notation.  
A {\it partition} $\gs=(\gs_1,\ldots,\gs_m)$ is a nonincreasing
sequence of positive integers, that is,

\centerline{$\gs_1\geq\cdots\geq\gs_m\geq 1,\ \hbox{for some}\ m\geq 1$.}

We say that {\it $\sigma$ is a partition of $n$} (in short, $\sigma\vdash n$)  if
$ \gs_1+\cdots+\gs_m=n$.
For $\sigma\vdash n$, let  $C_\sigma$ be the conjugacy class of all permutations
$\tau\in S_n$  with cycle lengths 
$\sigma_1,\,\sigma_2,\ldots,\sigma_m$, 
let $|C_\sigma|$ be its cardinality, and
let $\sgn(\sigma)$
be the common sign of the elements of $C_{\sigma}$.

For any $k\geq 1$, let $N_k(x_1,\dots,x_n)=x_1^k+\cdots +x_n^k$ 
be the {\it $k^{th}$-Newton polynomial} in $n$ indeterminates $x_1,\ldots,x_n$. For an arbitrary partition 
$\sigma=(\gs_1,\ldots,\gs_m)\vdash n$, set

\centerline{$N_\gs(x)=N_{\gs_1}(x)N_{\gs_2}(x)\cdots N_{\gs_m}(x)$.}

\bigskip

The following identity is usually called the
{\it Girard-Newton formula}.

\begin{lemma}\label{id1} We have

\centerline{${\sum_{\gs\ \vdash\ n}\, \hbox{sgn}(\gs)
\,|C_\gs|\,N_\gs(x)}
=n!\, x_1\cdots x_n$.}

\end{lemma}

\noindent
\begin{proof}  We recall the standard proof of this 17$^{th}$ century identity.
Let $V$ be the vector space $\K^n$ and 
let $h:V\to V$ be the diagonal matrix with diagonal entries
$x_1,\dots, x_n$. 

The symmetric group $S_n$ acts on $V^{\otimes n}$ by permutation of the factors. The Schur idempotent $(1/n!)\sum_{\tau\in S_{n}}\hbox{sgn}(\tau)\tau$ 
is a projection of $V^{\ot n}$ onto the $1$-dimen\-sional subspace  
$\bigwedge^nV$, so

\centerline{$\Tr (h^{\otimes n}\circ
\sum_{\tau\in S_{n}}\hbox{sgn}(\tau)\tau)=
n! x_1\cdots x_n$.}

Note that for $\sigma\vdash n$ and $\tau\in C_\sigma$, we have $\Tr (h^{\otimes n}\circ\tau)=N_\sigma(x)$. It follows that

\centerline{$\Tr (h^{\otimes n}\circ
\sum_{\sigma\in S_{n}}\hbox{sgn}(\sigma )\gs)=
\sum_{\sigma\vdash n}\, \hbox{sgn}(\gs)
\,|C_\gs|\,N_\gs(x)$,}

\noindent from which the lemma follows.
\end{proof}

\bigskip

Let $A:=\K[Y_1,Y_2,\dots]$ be the polynomial algebra in
infinitely many variables $Y_1, Y_2, \dots$. For any partition 
$\sigma=(\sigma_1,\dots,\sigma_\ell)$ of $n$, set 

\centerline{$Y_{\sigma}=Y_{\sigma_1}Y_{\sigma_2}\cdots Y_{\sigma_\ell}$.} 

\noindent Let $t$ be an indeterminate. In $A[[t]]$, the formal expression

\centerline{$\exp(-\sum_{k\geq 1} Y_k t^k/k)$}

\noindent can be expanded as $\sum_{k\geq 0}\,a_k t^k$, where all $a_k$ belongs to $A$.

\begin{lemma}\label{identity} We have
$a_n=((-1)^n/n!)\, {\sum_{\gs\ \vdash\ n} \hbox{sgn}(\gs)|C_\gs|Y_\gs}.$
\end{lemma}

\noindent
\begin{proof}  Both sides of the identity are polynomials
involving only the variables $Y_1,\dots, Y_n$. 
We will prove it for the elements $Y_k=N_k=x_1^k+\cdots + x_n^k$
of the algebra  $A'=\K[x_1,\dots,x_n]$. This is sufficient since
$N_1,\dots,N_n$ are algebraically independent.  Note that
\begin{align*}
-\sum_{k\geq 1} N_k t^k/k&=
-\sum_{k\geq 1}\sum_{1\leq i\leq n} x_i^k t^k/k
=-\sum_{1\leq i\leq n}\sum_{k\geq 1} x_i^k t^k/k
=\sum_{1\leq i\leq n} \log(1-tx_i).
\end{align*}
Hence
$$\exp(-\sum_{k\geq 1} N_k t^k/k)=\prod_{1\leq i\leq n} (1-tx_i).$$
It follows that $a_n=(-1)^n x_1\cdots x_n$, and
the identity now follows from  Lemma \ref{id1}.
\end{proof}

\subsection{The Dominance criterion}

After this combinatorial digression, we now answer the question about dominant $J$-spaces.
Let $(V,\rho)$ be a $J$-space and let $a\in J$. Since $\{a^n,a^m\}=0$ for any positive
integers $n$ and $m$, the operators $(\rho(a^n))_{n>0}$ commute with each other.  For any partition $\sigma=(\sigma_1,\sigma_2,\dots,\gs_m)$, we define
$\rho_\sigma(a):=\rho(a^{\sigma_1})\rho(a^{\sigma_2})\cdots\rho(a^{\sigma_m})$.

\begin{thm}\label{dominant}
A $J$-space $(V,\rho)$ of level $n$ is dominant  if and only if

\centerline{$\displaystyle{\sum_{\gs\ \vdash\ n+1} \hbox{sgn}(\gs)|C_\gs|\rho_\gs(a)=0}$,}

\noindent for all $a\in J$.
\end{thm}

\noindent
    \begin{proof} For any integer $k$,  we denote by
$M_k(V)$, $Z_k(V)$, and $\Delta_k(V)$ the 
corresponding eigenspaces in $M(V)$, $Z(V)$, and $\Delta(V)$.

By  commutativity of $\fG_{-2}$, the enveloping
algebra $U(\fG_{-2})$ is the symmetric algebra 
$S\fG_{-2}$, and
by the Poincar\'e-Birkhoff-Witt Theorem, 
$M(V)=U(\fG_{-2}) V$. It follows that
    
\centerline{ $M_{n-2k}(V)=S^{k}(\fG_{-2}) V$,}
    
\noindent for any integer $k$. Since
$\Delta(V)$ is an $\SL_2$-module and
$\Delta_{n+2}(V)=0$, we deduce that
$\Delta_{-n-2}(V)=0$. It follows that
$Z_{-n-2}(V)=M_{-n-2}(V)$ and therefore

\centerline{
$Z_n(V)=e^{n+1}M_{-n-2}(V)=
e^{n+1}S^{n+1}(\fG_{-2})V$.}

However, $Z_n(V)$ is the kernel of the map
$V\to\Delta(V)$ and  $S^{n+1}(\fG_{-2})$ is spanned by the elements $f(a)^{n+1}$ with $a\in J$. Therefore, $V$ is dominant if and only if

\centerline{$e^{n+1}f(a)^{n+1} m=0$,} 

\noindent for all $a\in J$ and $m\in V$. This condition is equivalent to

\centerline{$\pi(e^{n+1}f(a)^{n+1}) m=0$.} 

By Lemma \ref{garland},
$\pi(e^{n+1}f(a)^{n+1})$ is the coefficient of $t^{n+1}$
of the formal series

\centerline{$r\, \exp\left(-\sum_{s=1}^\infty \frac{h(a^s)}{s}t^s\right)$ for some $r\in\Q\setminus 0$,}

\noindent which, by Lemma \ref{identity}, is equal to 

\centerline{$r'\,\,{\sum_{\gs\ \vdash\ n+1} \hbox{sgn}(\gs)|C_\gs|\rho_\gs(a)},$ for some $r'\in\Q\setminus 0$.}

\noindent Hence $V$ is dominant if and only if 

$$\sum_{\gs\ \vdash\ n+1} \hbox{sgn}(\gs)|C_\gs|\rho_\gs(a)=0,$$

\noindent for all $a\in J$.\end{proof}

\subsection{The algebra ${\mathcal U}_n(J)$}\label{U_n}

For any partition
$\sigma=(\sigma_1,\sigma_2,\ldots,\gs_m)$ and $a\in J$, let
$h_{\sigma}(a)\in U(\fG_0)$ be the element defined by:

\centerline{$h_\sigma(a)=h(a^{\sigma_1})h(a^{\sigma_2})\cdots h(a^{\gs_m})$.}

\noindent For 
$n\geq 0$, let $I_n$ be the two-sided ideal of 
$U(\fG_0)$ generated by the elements
$h-n$ and 

\begin{equation}{\displaystyle{\sum_{\sigma \vdash\ n+1} \sgn(\gs)
|C_\gs|\, h_\gs(a)},}\label{bloup}
\end{equation}

\noindent for all $a\in J$. Let

\centerline{${\mathcal U}_n(J)=U(\fG_0)/I_n$.}

\noindent We write $\rho_n:J\to {\mathcal U}_n(J)$ for the natural map
$a\mapsto h(a)\mod I_n$.

\subsection{Examples}\label{examples}

\noindent Theorem \ref{dominant} states that the dominant
$J$-spaces of level $n$ are exactly the ${\mathcal U}_n(J)$-modules.
Here are some examples.

\begin{enumerate}

\item[(a)] It is clear that the algebra $\mathcal{U} _1(J)$ is isomorphic to the unital special universal envelope $S_1(J)$, defined in \cite[Chap.~II.1]{Jacobson}.

  \item[(b)]  The multiplication algebra of $J$ is the subalgebra of $\End(J)$ generated by the multiplications $L_a$. It satisfies the obvious identities
  \begin{enumerate}
    \item[{\rm (b.1)}]$L_1=1$
 \item[{\rm (b.2)}] $[L_{a^2},L_a]=0$
 \item[{\rm (b.3)}] $L_b L_a L_c + L_c L_a L_b + L_{(bc)a} = L_a L_{bc} + L_c L_{ab} + L_b L_{ac}.$
\end{enumerate}
  Following \cite[Chap.~II.9 and II.11]{Jacobson}, the unital universal multiplication envelope, denoted by $U_1(J)$ in \cite{Jacobson}, is the unital associative algebra generated by symbols $L_a$ and defined by the relations (b.1-b.3). In fact, 	if $M$ is a $U_1(J)$-module, then the split null extension $J\oplus M$ is a Jordan algebra, which is unital by Relation (b.1).

  Relation (b.3) can be split into the following two identities:
        \begin{enumerate}
          \item[{\rm (b.3.1)}] $2 (L_a)^3+ L_a^3= 3 L_a L_{a^2}$
            \item[{\rm (b.3.2)}] $[[L_a, L_b], L_c]= L_{\partial_{a,b}c}$.
 \end{enumerate}
	There is thus an isomorphism from $\mathcal{U}_2(J)$ to the unital universal multiplication envelope $U_1(J)$, sending $h(a)\mapsto 2 L_a$.


\item[(c)] For $J=\K[a]$, the proof of Theorem \ref{dominant}
shows that 

\centerline{${\mathcal U}_n(J)\simeq \K[x_1,\dots,x_n]^{S_n}$.}

\noindent The isomorphism  sends
$h(a^k)$ to $N_k$ for any $k$. 

\item[(d)] Let $V$ and $V'$ be dominant $J$-spaces of levels $p$ and $q$. Since $V\otimes V'$ is the highest component of the
$(\fG,\SL_2)$-module $\Delta(V)\otimes\Delta(V')$, the $J$-space
 $V\otimes V'$ is dominant of level $p+q$. 
Hence there is an algebra morphism

\centerline{$\Delta_{p,q}: {\mathcal U}_{p+q}(J)\to
{\mathcal U}_{p}(J)\otimes {\mathcal U}_{q}(J)$.}

\item[(e)] Let $J$ be a unital commutative associative algebra. In general, $\mathcal{U}_n(J)$ is not commutative, even for $n = 1$. The simplest example is the unital Jordan algebra $J = \K \oplus V$ , where $V$ is a $2$-dimensional ideal with $V^2 = 0$. It is clear that $\mathcal{U}_1(J)$ is the exterior algebra $\wedge V$, which is not commutative. Using the coproducts defined in Example (c), it follows that $\mathcal{U}_n(J)$ is not commutative for any $n\geq 1$.

Assume now that $J$ is an arbitrary unital commutative associative algebra containing an ideal $\mathfrak{m}$ with $\dim (\mathfrak{m}/\mathfrak{m}^2) \geq 2$. Then $\K \oplus V$ is a quotient of $J$, and therefore the algebra $\mathcal{U}_n(J)$ is not commutative for any $n\geq 1$.

\item[(f)] There is a classical notion of Jordan modules, see \cite{Jacobson}, and small variations depending on whether the Jordan algebras are assumed to be unital (as in the present paper) or not. For clarity, we call the $\mathcal{U}_0(J)$-modules the {\em trivial Jordan modules}, the $\mathcal{U}_1(J)$-modules the {\em $1$-sided Jordan modules}, and the $\mathcal{U}_2(J)$-modules the {\em $2$-sided Jordan modules.} In general, $\mathcal{U}_n(J)$-modules are not Jordan modules of any of the previous types for $n\geq 3$.

 Indeed, for $n\geq 3$, $\mathcal{U}_n(\K[a])$ is not isomorphic to any $\mathcal{U} _{\ell}(J)$ for $\ell\leq 2$. In fact, $J$ would then be a commutative associative algebra of Gelfand-Kirillov dimension at least $2$.  Example (e) would then show that $\mathcal{U}_{\ell}(J)$ is not commutative.

The terminology {\em $J$-spaces} and {\em dominant $J$-spaces} has been chosen to avoid confusion with the classical notion of Jordan modules.



\item[(g)] Set $\tilde J=J*\K[a]$, where $*$ denotes the free product.
  The Lie algebra $\fg=\widehat{\fsl}_2(\tilde J)$ has a natural grading $\fg=\bigoplus_{n\in\mathbb{Z}_+} \fg_n$, with $\fg_0=\widehat\fsl_2(J)$ and $x(a)$ homogenous of degree $1$ for all $x\in\fsl_2$.  Let $\mathcal {M}_{\bf T} (\widehat{\fsl}_2(J))$ be the category of $\widehat{\fsl}_2(J)$-modules which decompose, as $\fsl_2$ -modules, into sums of trivial and adjoint modules. By the universal property of the Joseph functor, it is clear that $\Delta(\mathcal{U}_2(J))$ is the rank $1$ free module in the category  $\mathcal {M}_{\bf T} (\widehat{\fsl}_2(J))$. It follows easily that $\Delta(\mathcal{U}_2(J))$ is isomorphic to $\fg_1$, cf.~\cite[Section 5.6]{KM}.


\item[(h)] Viewing $J$  as a $J$-space of level $2$, there is a natural $\fG$-equivariant homomorphism
  $\pi:\Delta(J)\to\fG$.  The map $\pi$ defines a (left) Leibniz bracket on $\Delta(J)$ by $[a,b]:=\pi(a).b$ for all $a,b\in \Delta(J)$, where $\pi(a).b$ denotes the $\fG$-module action of $\pi(a)$ on $b$.  Indeed, $\Delta(J)$ is the universal central cover of
$\fG$ in the category of Leibniz algebras.

\item[(i)] Let $J=\K[t,t^{-1}]$. The Lie algebra $\fG_0$
  has a $1$-dimensional centre $\fz$.  By an easy argument, $\fz$ acts trivially on any Weyl module. Hence the
map $\fG_0\to {\mathcal U}_{n}(J)$ is not injective and
$\bigcap I_n\neq 0$. Indeed, it is easy to see that
${\mathcal U}_{n}(J)\simeq \K[x_1^{\pm1},\dots, x_n^{\pm1}]^{S_n}$
for any $n\geq 0$.


\item[(j)] Let $A$ be an associative commutative algebra.
In their interesting paper 
\cite{Chari-Fourier-Khandai}, V.~Chari, G.~Fourier, and T.~Khandai  investigate representations of the Lie algebra
$\fg\otimes A$, where $\fg$ is a finite dimensional simple Lie algebra. For $\fg=\fsl_2$, they define a certain
commutative associative  algebra $A_{n\rho}$. It can be shown that
$A_{n\rho}$ is the maximal commutative quotient of the algebra 
${\mathcal U}_n(A)$.

\item[(k)] Let $X$ be an affine algebraic variety, and consider its ring of functions $\K[X]$ as a Jordan algebra. When $\dim X\geq 2$, Example (e) shows that
  $\mathcal{U}_n(\K[X])$ is not commutative for any $n \neq 0$. Similarly, for singular curves $X$, the algebra $\mathcal{U}_n(\K[X])$ is not commutative.  Assume now that $X$ is a smooth curve. Then Example (i) shows that all algebras $\mathcal{U}_n(\K[X])$ are commutative whenever $\hbox{H}^1_{DR}(X)$ is generated by logarithmic differentials, that is, when $X$ has genus $0$.  This raises the following question:
\begin{center}
{\em For arbitrary smooth curves $X$, are the algebras $\mathcal{U}_n(\K[X])$ commutative?}
\end{center}

\bigskip

\end{enumerate}

\subsection{The algebras ${\mathcal U}_n(J)$ for finite dimensional Jordan algebras $J$}

Let $\fg$ be a finite dimensional Lie algebra, and let $I\subseteq U(\fg)$ be a left ideal. Let $\K=U_0(\fg)\subseteq U_1(\fg)\subseteq 
U_2(\fg)\subseteq\cdots$ be the canonical filtration of the enveloping algebra.

\begin{lemma}\label{involutility} Let $n$ be an integer and set
${\mathcal X}=\{x\in\fg\mid x^n\in I+U_{n-1}(\fg)\}$.

If ${\mathcal X}$ generates the Lie algebra $\fg$, then
$I$ has finite codimension.
\end{lemma}

\begin{proof} Let $I^{gr}\subset S\fg$ be the graded ideal associated to $I$, and let $\sqrt{I^{gr}}$ be its radical. By definition,  $\sqrt{I^{gr}}$ contains  
$\mathcal{X}$. By
Gabber's Involutivity Theorem \cite{gabber}, $\sqrt{I^{gr}}$ is stable under the Poisson bracket of $U(\fg)$, hence
$\sqrt{I^{gr}}$ is the  codimension $1$
ideal $\fg S\fg$. Therefore $I^{gr}$ and
$I$ have finite codimension.
\end{proof}

\begin{lemma} \label{U_nbasic} Assume that $\dim J<\infty$. 

\begin{enumerate}

\item[(a)] For any finite dimensional 
$(\fB,\B)$-module $V$,
the Weyl module $\Delta(V)$ is finite dimensional.

\item[(b)] The associative algebra ${\mathcal U}_n(J)$ is finite dimensional for any $n\geq 0$.
\end{enumerate}
\end{lemma}

\noindent{\it Proof of (a).} As an $S\fG_{-2}$-module,
 $\Delta(V)$ is generated by the image of $V$ in
 $\Delta(V)$. Hence
each eigenspace $\Delta_k(V)$ is finite dimensional and
the eigenvalues are bounded by some integer $N$.
By Lemma \ref{universal}, $\Delta(V)$ is an $\SL_2$-module, so
 its eigenvalues are a finite set of integers in
$[-N,N]$. Therefore $\dim\,\Delta(V)<\infty$.

\smallskip
\noindent{\it Proof of (b).} For  $n\geq 0$, recall that
${\mathcal U}_n(J)=U(\fG_0)/I_n$ where the ideal 
$I_n$ is defined as in Subsection \ref{U_n}. Set
${\mathcal X}=\{x\in\fG_0\mid x^{n+1}\in I+U_{n}(\fG_0)\}$.
By Equation \ref{bloup},
$h(a)^{n+1}$ is a polynomial of degree
at most $n$ in $h(a), h(a^2),\ldots, h(a^{n+1})$ modulo $I_n$,
which implies that $h(a)$ belongs to ${\mathcal X}$. 
Since  $\fG_0=h(J)+[h(J),h(J)]$, the Lie algebra
$\fG_0$ is finite dimensional and generated by ${\mathcal X}$.
By Lemma  \ref{involutility}, the associative algebra
${\mathcal U}_n(J)$ is finite dimensional.\qed

\begin{lemma} Assume than $J$ is simple and finite dimensional

\begin{enumerate}

\item[(a)] The Lie algebra $\widehat\fsl_2(J)$ is simple and
$\Delta(V)$ is a simple $\widehat\fsl_2(J)$-module for 
every simple dominant $J$-space $V$.

\item[(b)] The associative algebras ${\mathcal U}_n(J)$ are finite dimensional and semisimple.

\item[(c)] The map $V\mapsto \Delta(V)$ is 
a bijection from the class of simple dominant $J$-spaces
to the class 
of  finite dimensional simple  
$\widehat\fsl_2(J)$-modules $L$. The inverse bijection is
$L\mapsto H^0(\fG_2,L)$.
\end{enumerate}
\end{lemma}

\begin{proof} It is well known that the Lie algebra
$TKK(J)$ is simple. Since any central extension of a finite dimensonal simple 
Lie algebra is trivial, we have 
$\widehat\fsl_2(J)=TKK(J)$, which proves Assertion (a).

It is clear that the $\fG_0$-module $\fG_2$ is simple and faithful. 
Thus the center of $\fG_0$ is the 
1-dimensional subspace $\K h$,
$\fG_0=\K h\oplus [\fG_0,\fG_0]$, and the
Lie algebra $[\fG_0,\fG_0]$ is semisimple. By Lemma
\ref{U_nbasic}(b), the algebra $\mathcal{U}_n(J)$ is a finite 
dimensional quotient of $U([\fG_0,\fG_0])$, and therefore
$\mathcal{U}_n(J)$ is semisimple.

Assertion (c) follows from the usual highest weight theory.\end{proof}

\subsection{The algebras ${\mathcal U}_n(\A)$ for the Albert  algebra $\A$}

We specialize the results of the previous subsection to  the  Albert algebra $\A$ of $3\times 3$ Hermitian matrices over the split octonion algebra. The Lie algebra $\fG:=\widehat\fsl_2(\A)=TKK(\A)$ is a split Lie algebra of type $E_7$ \cite{T}. Let 
$\fh\subset\fG_0$
be a split Cartan subalgebra. There is  a 
basis $R$  of the root system such that all roots in
$\fG_2$ are positive. We can identify $R$ with the 
basis $\{\alpha_1,\cdots,\alpha_7\}$ of the root system of $E_7$ defined in
\cite[ch. IV section 4.11]{Bou}. Since
$\{\alpha_1,\cdots,\alpha_6\}$ is a basis of the root system of $[\fG_0,\fG_0]\simeq E_6$, we have

\centerline{$\alpha_1(h)=\cdots=\alpha_6(h)=0$
and $\alpha_7(h)=2$.}
 
\noindent It follows that 

\begin{equation}\label{h}
h=2h_1+3h_2+4h_3+6h_4+5h_5+4h_6+3h_7,
\end{equation}

\noindent where $h_1,\dots,h_7$ is the corresponding basis of the dual root
system. For any integral dominant weight $\lambda$, let
$L(\lambda)$ be the simple $\fG$-module with highest weight $\lambda$. Then $H^0(\fG_2, L(\lambda))$ is a dominant $\A$-space
of level $\lambda(h)$.

Albert has proved that  $\A$ is exceptional and simple \cite{Albert}, so ${\mathcal U}_1(\A)=0$. This 
also follows from the fact that no coefficient in Equation \ref{h} is $1$. When 
$\lambda$ is the fundamental weight $\omega_2$ or $\omega_7$, the dominant 
$\A$-space $H^0(\fG_2, L(\lambda))$ is of level 3.
We conclude:

\begin{lemma}\label{Albert}
 We have  ${\mathcal U}_3(\A)\neq 0$.\qed
 \end{lemma}

\section{ Injectivity of the maps $\rho_n:J\to {\mathcal U}_{n}(J)$ for $n\geq 4$}

We have seen in  Example \ref{examples}  (h) that the map 
$\fG_0\to {\mathcal U}_{n}(J)$ is not always injective, even for $n\gg 0$. By contrast, the map $\rho_2$ is obviously injective,
and it is easy to deduce that $\rho_n$ is injective for any
even integer $n\geq 2$. 

We now show that  $\rho_n:J\to {\mathcal U}_{n}(J)$
is injective, whenever $n\geq 4$. When $n$ is odd, our proof is
based on Zelmanov's classification theorem.

\subsection{A technical lemma}

Recall that the algebra $\{0\}$ is called {\it trivial}.

\begin{lemma}\label{rhopq} 
Assume that the algebra ${\mathcal U}_{p}(J)$ is not trivial,
and $\rho_q$ is injective
for some nonnegative integers $p$ and $q$.

Then $\rho_{p+q}$ is injective
\end{lemma}

\begin{proof} Let 

\centerline{$\Delta_{p,q}: {\mathcal U}_{p+q}(J)\to
{\mathcal U}_{p}(J)\otimes {\mathcal U}_{q}(J)$,}

\noindent be the algebra morphism defined in
Example \ref{examples}(f).
 We show that $\Delta_{p,q}\circ \rho_{p+q}(a)$ is not zero
 for  any $a\in J\setminus 0$. We have
$\Delta_{p,q}\circ \rho_{p+q}(1)=(p+q)1\otimes 1$, so
$\Delta_{p,q}\circ \rho_{p+q}(a)\neq 0$ if $a$ is a scalar.
Otherwise, we have $\Delta_{p,q}\circ \rho_{p+q}(a)
=\rho_p(a)\otimes 1+1\otimes \rho_q(a)$, where
$\rho_q(a)$ and $1$ are not proportional.
Thus  
$1\otimes \rho_q(a)$ is not proportional to
$\rho_p(a)\otimes 1$. So $\Delta_{p,q}\circ \rho_{p+q}(a)\neq 0$,
which proves that $\rho_{p+q}$ is injective.
\end{proof}

\subsection{Zelmanov's classification theorem}

\noindent The center of a   simple Jordan algebra $J$ is a field
$L$. 
The Jordan algebra $J$ is called an {\it Albert algebra} if
$\overline{L}\otimes J$ is the Albert algebra over the algebraic closure $\overline{L}$ of $L$.  The following theorem is proved in \cite{Zelmanov80}, \cite{Zelmanov83}.

\begin{ZClassthm}\label{Zclas} Any simple Jordan algebra is either special
or is an Albert algebra.
\end{ZClassthm}

\begin{cor}\label{U3} For any Jordan algebra $J\neq 0$, the algebra 
${\mathcal U}_3(J)$ is not trivial.
\end{cor}

\begin{proof} By Zorn's Lemma, $J$ has a simple quotient.
Thus it is enough to prove the assertion when $J$ is simple.

If $J$ is special, the map 
$\rho_3$ is injective by Lemma \ref{rhopq}. Otherwise, $J$ is an Albert algebra by Zelmanov's Classification Theorem \ref{Zclas}, and it
follows from Lemma \ref{Albert} that ${\mathcal U}_3(J)$ is not trivial.
\end{proof}

\subsection{The injectivity theorem}

\begin{thm}\label{injectivity} The map $\rho_n:J\to
{\mathcal U}_n(J)$ is injective for any $n\geq 4$.
\end{thm}

\begin{proof} Since $\rho_2$ is injective, an iterative use
of Lemma \ref{rhopq} shows that $\rho_n$ is injective for any
even integer $n\geq 2$. 

For an odd integer $n\geq 5$,
the algebra ${\mathcal U}_3(J)$ is not trivial and it has been proved that $\rho_{n-3}$ is injective. So by Lemma
\ref{rhopq}, the map $\rho_n$ is injective.
\end{proof}

This leaves open two questions: is there an elementary proof of Theorem \ref{injectivity}? Is it always true that $\rho_3$ is injective?

\section{Finite generation, finite presentation, and $FP_\infty$}\label{FP_infinity}

As a motivation for considering the $FP_\infty$ condition, we first state
two finiteness results about Jordan algebras. Then we show that if $\widehat{\mathfrak{sl}}_2(J(D+1))$ is $FP_\infty$, then the graded dimension conjecture holds for the free Jordan algebra $J(D)$.


\subsection{Finite generation and finite presentation}

We now state two finiteness results.

\begin{thm}\label{finitegen} The associative
algebra  ${\mathcal U}_n(J)$ is finitely generated, for all finitely generated Jordan algebras $J$.
\end{thm}

This generalizes the following known results for
a  Jordan algebra  generated by a finite set  $x_1,\dots,x_n$
of elements:

\begin{enumerate}

\item[(a)] The universal special unital envelope
${\mathcal U}_1(J)$ is generated by the the elements
 $\rho_1(x_1)$, $\dots$, $\rho_1(x_n)$. 

\item[(b)] The unital univeral multiplication envelope
${\mathcal U}_2(J)$ is  generated by 
the elements $\rho_2(x_i)$ and $\rho_2(x_i\cdot x_j)$, for $1\leq i\leq j\leq n$.

\item[(c)] If $J$ is an associative algebra $A$, it has been proved in \cite{Chari-Fourier-Khandai} that the maximal commutative quotient 
$A_{n\rho}$ of ${\mathcal U}_n(A)$ is finitely generated.

\end{enumerate}

Another finiteness result is the following:

\begin{thm}\label{finitepres}  The Lie algebra $\widehat\fsl_2(J)$ is finitely presented, for all finitely presented Jordan algebras $J$.
\end{thm}

The proof of both theorems can be reduced to  free Jordan algebras.
The proofs are postponed to  Subsections \ref{fgproof} and \ref{fpproof}.

\subsection{Koszul lemma}

Let $\fg=\fs\ltimes\fu$ be a Lie algebra which is a semi-direct product.  We write $H_*(\fg,\fs)$ for the {\it relative homology}, that is, the homology of the Chevalley-Eilenberg complex of coinvariants $\bigwedge^p(\fg/\fs)\big/\big(\fs\cdot\bigwedge^p(\fg/\fs)\big)$.  See \cite{Ko} or \cite{kumar} for details.  The following lemma is essentially due to Koszul 
\cite[Theorem 17.3]{Ko}.

\begin{lemma}\label{Koszul} Assume that $\fs$ is a finite dimensional simple Lie algebra and $\fu$ is locally finite as an $\fs$-module.
Then $H_*(\fg)$ is isomorphic to $H_*(\fs)\otimes H_*(\fg,\fs)$.
\end{lemma}

In fact, Theorem 17.3 in \cite {Ko} only involves finite dimensional Lie algebras, but the proof in our setting is exactly the same.

\subsection{Action of the Witt algebra}

Let $J(D)$ be the free Jordan algebra on the
generators $x_1,\dots,x_D$. For $n\geq -1$, set
$L_n:=\sum_{i=1}^D\,x_i^{n+1}{\partial\over \partial x_i}$. We have

\centerline{$[L_n,L_m]=(m-n) L_{m+n}$,}

\noindent so the linear span $W$ of the $L_n$ is
a Lie algebra of derivations of $J$ isomorphic to
the Witt algebra $\Der (\K[t])$.

\begin{lemma}\label{Witt} For any $k\geq 4$, the space 
$H_k(\widehat{\fsl}_2(J(D)))$ is either $\{0\}$
or is infinite dimensional.
\end{lemma}

\begin{proof} Assume $k\geq 4$. 
The derivation $L_0$ acts on
the component $J_n(D)$ by multiplication by $n$.
 Since all eigenvalues of $L_0$ on
$\Lambda^k(\widehat{\fsl}_2(J(D)))$ are positive integers, it follows that 

\centerline{
$L_0.H_k(\widehat{\fsl}_2(J(D)))=H_k(\widehat{\fsl}_2(J(D)))$.}

\noindent Hence $H_k(\widehat{\fsl}_2(J(D)))$ is either
zero, or is a nontrivial $W$-module. In the latter
case $H_k(\widehat{\fsl}_2(J(D)))$ is infinite dimensional since 
$W$ is a simple infinite dimensional Lie algebra.
\end{proof}

\subsection{The $FP_\infty$-property}

By definition, a  Lie algebra $\fg$ is of type $FP_n$ if the trivial module
$\K$ admits a projective resolution

\centerline{$0\lto \K\lto C_0\lto C_1\lto C_2\lto\cdots $}

\noindent where the modules  $C_i$ are finitely generated for all $i\leq n$. We say that $\fg$ is 
of type $FP_\infty$ if it is of type $FP_n$ for all integers $n$.

\begin{proposition} The Lie algebra
$\widehat\fsl_2(J(1))$ is of type $FP_\infty$.
\end{proposition}

\begin{proof} 
By the main result
of Garland and Lepowsky's paper \cite{GL} there is a resolution of the trivial $\fsl_2(t\K[t])$-module $\K$

\centerline{$0\lto \K\lto C_0\lto C_1\lto C_2\dots $}

\noindent where $C_n$ is a free $U(\fsl_2(t\K[t]))$-module
of rank $2n+1$ for any $n\geq 0$, so
$\fsl_2(t\K[t])$ is of type $FP_\infty$. Since

\centerline{
$\widehat\fsl_2(J(1))=\fsl_2(J(1))=
\fsl_2\ltimes \fsl_2(t\K[t])$}

\noindent the Lie algebra $\widehat\fsl_2(J(1))$
is also of type $FP_\infty$.
\end{proof}

\subsection{The $FP_\infty$ condition}

By the previous proposition, the Lie algebra
$\widehat\fsl_2(J(1))$ is of type $FP_\infty$.
Moreover, by  Theorem \ref{finitepres}, the Lie algebra
$\widehat\fsl_2(J(D))$ is of type $FP_2$ for any $D$. Also Theorems \ref{finitegen} and \ref{finitedim} suggest that $\widehat\fsl_2(J(D))$
enjoys some further finiteness properties.
So it is natural to suspect that 
$\widehat\fsl_2(J(D))$ is of type $FP_\infty$.

\begin{thm}\label{FPhypothesis} Assume that
$\widehat\fsl_2(J(D+1))$ is of type $FP_\infty$.
Then the conjectural formula of \cite{KM} for 
$\dim J_n(D)$ holds.
\end{thm}

\begin{proof} By  hypothesis, each homology space $H_k(\widehat{\fsl}_2(J(D+1)))$ is finite dimensional. So, by Lemma \ref{Witt} we have

\centerline{$H_k(\widehat{\fsl}_2(J(D+1)))=0$,}

\noindent for any $k\geq 4$. By Koszul's
Lemma \ref{Koszul},  there is an isomorphism

\centerline{
$H_*(\fsl_2(J(D+1)))=H^*(\fsl_2)\otimes 
H_*(\fsl_2(J(D+1)),\fsl_2)$.}

\noindent Since $H_3(\fsl_2)$ is nonzero, we deduce that

\centerline{$H_k(\widehat{\fsl}_2(J(D+1)),\fsl_2)=0$,}

\noindent for any positive integer $k$. 
Thus by \cite[Theorem 2]{KM}, this implies 
the conjectural formula for 
$\dim J_n(D)$.
\end{proof}

\noindent{\it Remark.} Our notation here is different from \cite{KM}. In \cite {KM}, the free nonunital Jordan algebra  is denoted $J(D)$,
while the unital free Jordan algebra is denoted
$J^u(D)$. Also it should be noted that

\centerline{$H_*(\widehat{\fsl}_2(J(D)),\fsl_2)
=H_*(\widehat{\fsl}_2(J_+(D)))^{\fsl_2}$,}

\noindent where $J_+(D)$ is the free 
nonunital Jordan algebra on $D$ generators.

\smallskip
\noindent
    {\it Recent developments.} In a very recent arXiv preprint \cite{DH} posted when the present paper was under final revisions, Dotsenko and Hentzel have provided detailed computational data showing that the dimension formula conjectured in \cite{KM} does not always hold.  Theorem \ref{FPhypothesis} thus admits the following, perhaps surprising, corollary:

    \begin{cor}\label{notFPinfinity} For $D\geq 3$, the Lie algebra $\widehat{\fsl}_2(J(D))$ is {\it not} of type $FP_\infty$.\qed
      \end{cor}

\vskip1cm
\centerline{\Large \bf PART B: Good Filtrations in LA+ categories}
\vskip1cm

\section{LA+ categories}\label{defLA+}

\subsection{ Definition of  LA categories}\label{LA}

Let $\mathcal{C}$ be $\K$-linear abelian  category.
We say that $\mathcal{C}$ is a {\it locally artinian
category} (LA category) if 

\begin{enumerate} 

\item[{\rm(LA1)}]  $\mathcal{C}$ admits
direct limits  
$\varinjlim\, A_n$
for any directed family
$$A_0\to A_1\to A_2\to\cdots $$
of objects indexed by $\Z_{\geq 0}$,

\item[(LA2)] the functor $ \varinjlim$
 is exact, and 

\item[(LA3)] every object is the directed union 
$\bigcup_{n\geq 0} A_n$  of a directed family of subobjects $A_n$ of finite length.

\end{enumerate}

\noindent By definition, a {\it directed union} $\bigcup_{n\geq 0} A_n$ will be a direct limit 
$\varinjlim\,A_n$ of a directed family $(A_n)_{n\geq 0}$
where all maps $A_n\to A_{n+1}$ are injective. 
For concreteness, all directed unions will be indexed by $\Z_{\geq 0}$, so Axiom  (LA3) implies that all objects have finite or countable length. 

In this paper, we handle inductive  limits cautiously. In fact,
from the time of the groundbreaking paper of Cline, Parshall, and Scott \cite{cps88}, it has been known that the paper \cite{Roos61} about the Grothendieck conditions is not fully correct \cite{D02,Nee02}. These conditions are implied by our Axiom (LA2).

\subsection{ Definition of LA+ categories}\label{LA+}

Let ${\mathcal C}$ be an LA category.
Recall that the {\it socle} of an object
$X$, denoted by $\Soc(X)$, is the largest semi-simple object in $X$. Indeed the axioms of an LA category ensure the existence of such an object. An
{\it injective envelope} of an object 
$X$ is an injective object $I\supseteq X$ such that 
$\Soc\,I=\Soc\,X$.  The first two axioms of an LA+ category deal with simple objects and their socles:

\begin{enumerate} 

\item[(AX1)] There is a set  
$\{L(n)\mid n\in\Z_{\geq 0}\}$ of objects
in ${\mathcal C}$ which is a complete list of nonisomorphic simple objects. 

\item[(AX2)] We assume that 
$\End_{\mathcal C}(L(n))=\K$ and that $L(n)$ admits 
an injective envelope $I(n)$ in $\mathcal{C}$, for all 
$n\geq 0$. 

\end{enumerate}

\noindent Some
additional definitions are needed to state the third axiom.

A {\it filtration} ${\mathcal G}_*$ of an object $M$ is a directed family $({\mathcal G}_n M)_{n\geq 0}$ of subobjects for which $M$ is the directed union
$\bigcup_{n\geq 0}{\mathcal G}_n M$.  If ${\mathcal G}_n M=M$ for  $n\gg 0$, the filtration is said to be {\em finite};  otherwise, it is {\em infinite}.  For any $n\geq 0$, set 
$\overline{\mathcal G}_n M={\mathcal G}_n M/{\mathcal G}_{n-1} M$, with the convention that 
$\overline{\mathcal G}_0 M={\mathcal G}_0 M$.
The objects $\overline{\mathcal G}_n M$ are called the {\it successive quotients} of the filtration ${\mathcal G}_*$.

Let ${\mathcal C}(n)$ be the subcategory of objects $M$ which are constructed as  finite or infinite extensions between the modules $L(0), L(1),\ldots, L(n)$. That is, $M$
admits a filtration ${\mathcal G}_*$ whose successive quotients 
$\overline{\mathcal G}_m M$ are $0$ or 
$L(k)$ for some $k=k(m)\leq n$.

For
$M\in {\mathcal C}$, let ${\mathcal F}_n\,M$ be the biggest
subobject in ${\mathcal C}(n)$. The filtration

\centerline{${\mathcal F}_0 M\subseteq  
{\mathcal F}_1 M\subseteq {\mathcal F}_2 M \subseteq \cdots$}

\noindent is called the {\it canonical filtration} of $M$.
By definition, the {\it standard object} $\nabla(n)$ is the inverse image in $I(n)$ of ${\mathcal F}_{n-1}(\,I(n)/L(n))$.

\bigskip

\noindent
{\bf Remarks:} \begin{enumerate}

\item[(a)]
Axiom (AX1) provides a linear ordering 

\centerline{$L(0)\leq L(1)\leq L(2)\leq \cdots$}

\noindent of the simple objects $L(n)$, and the definitions of the canonical filtration
${\mathcal F}_*$ and the standard objects $\nabla(n)$ depend on the 
chosen order. 

\item[(b)] According to the context, the  standard objects are sometimes called costandard objects or dual of Weyl modules. We follow Donkin's notation \cite {D82}.
These objects are denoted $A(n)$ in \cite{cps88}.
\end{enumerate}

\noindent
Our last axiom is

\begin{enumerate}
\item[(AX3)] For any $n\geq 0$, the standard object $\nabla(n)$ has finite length.
\end{enumerate}

We say that the LA category ${\mathcal C}$ is a {\it LA category plus additional axioms} (for short, an {\em LA+ category}) if it satisfies the three axioms (AX1), (AX2) and (AX3).

\subsection{  Existence of injective envelopes in LA+ categories}

In our setting, we did not require that 
${\mathcal C}$ has enough injectives because this follows from the hypotheses, as is proved now. 

\begin{lemma}\label{env} Let ${\mathcal C}$ be an LA+ category.

\begin{enumerate}

\item[(a)] Let $(I_n)_{n\in\mathbb{Z}_+}$ be a family of injective objects in
${\mathcal C}$. Then 

\centerline{$\bigoplus_{n\geq 0}\,I_n$}

\noindent is an injective object. 
 
\item[(b)] Any object in ${\mathcal C}$ admits an injective envelope.
\end{enumerate}
\end{lemma}

\noindent{\it Proof of (a).} Set $I=\bigoplus_{n\geq 0}\,I_n$. Let $X$ be an object of
 ${\mathcal C}$, let $Y\subseteq X$ be a subobject,  and let 
 $g:Y\rightarrow I$ be a morphism. We claim that $g$ can be extended to a
 morphism $f:X\to I$.

 By Axiom (LA3), $X$ is a  directed union $\bigcup_{n\geq 0} X_n$
  of finite length subobjects. Set $Y_n=Y\cap X_n$ and 
 $I_{\leq n}=\oplus_{k\leq n} I_k$. We have

 \centerline{$g(Y_n)\subseteq I_{\leq \phi(n)}$,}
 
 \noindent for some nondecreasing function 
 $\phi:\Z_{\geq 0}\to\Z_{\geq 0}$. If needed, we can replace $I_n$ by
 $\bigoplus_{\phi(n-1)<k\leq \phi(n)} I_k$, to assume without loss of generality that 
 
 \centerline{$g(Y_n)\subseteq I_{\leq n}$.}
 
 \noindent Let $g_n:\,Y_n\rightarrow I_n$ be the projection of $g$ on $I_n$.
 Since $g_n(Y_{n-1})=0$ there is a morphism
  $f_n:X\to I_n$ such that $f_n\vert_{Y}=g_n$ and 
  $f_n\vert_{X_{n-1}}=0$.
 
Since the sum $f:=\sum_{n\geq 0}f_n$ is locally finite,
it defines a morphism $f:X\to I$ extending $g$,
 which proves the claim. Hence
 $I$ is injective.
 
 \smallskip
 
 \noindent{\it Proof of (b).} Let $X$ be an object in 
 ${\mathcal C}$. Its socle $\Soc(X)$  is a direct sum of at most
 $\aleph_0$ simple objects.  By Part (a), $\Soc(X)$ has an injective envelope, which is also an injective envelope of $X$. \qed
 
 \bigskip
 
\noindent{\bf Remarks:} \begin{enumerate}

\item[(a)] Lemma \ref{env} is applied to the concrete category
  ${\mathcal C}_{sm}(J)$ considered in Part B.  Even in the concrete category case, the proof does not simplify and remains the same as the formal proof given here. 

\item[(b)] For categories which are not locally artinian, Lemma \ref{env} (a)
 sometimes fails. Here is an easy example.

 Let $F$ be the free module of rank $1$, over the free associative algebra $A$ on $2$ generators, and let
$E\supseteq F$ be an injective module containing $F$.
Since $F$ contains a free submodule $F^{\aleph_0}$ of rank $\aleph_0$,
the module $E^{\aleph_0}$ is not injective. Indeed, the inclusion map
$F^{\aleph_0}\to E^{\aleph_0}$ does not extend to $F$.

\end{enumerate}

\subsection{ Derived functors and direct unions of objects}

 Let  $F:{\mathcal C}\to {\mathcal B}$
be a left exact functor from an LA+ category $\mathcal{C}$ to an abelian category $\mathcal{B}$ satisfying axioms (LA1) and (LA2).  Since $\mathcal{C}$ has enough injectives, 
right derived functors of $F$  are well defined.
We show that $R^k F$ behaves nicely with directed unions, whenever the functor $F=R^0F$ commutes with directed unions.

We start with a basic lemma of homological algebra.
Recall that a directed union is a direct limit of a directed family in which all maps are injective.

\begin{lemma}\label{reso} Let $A=\bigcup_{n=0}^\infty A_n$ be a directed union.  Then there is a commutative diagram

\hskip3cm $A_0\longrightarrow A_1\longrightarrow A_2\longrightarrow\cdots$
 
\hskip3cm  $\hskip2mm\downarrow \hskip10mm\downarrow
\hskip11mm \downarrow$

\hskip3.1cm $I_0\hskip1mm\longrightarrow I_1
\hskip1mm\longrightarrow I_2\longrightarrow\cdots$

\noindent satisfying the following conditions:

\begin{enumerate}

\item[(a)] all morphisms are injective,

\item[(b)] the objects $I_k$ are  injective, and

\item[(c)] all maps in the directed family

\centerline{$I_0/A_0\rightarrow I_1/A_1\rightarrow I_2/A_2\rightarrow\cdots$}

are injective.

\end{enumerate}

\end{lemma}

\begin{proof} For each $k\geq 0$, 
 let $A/A_{k-1}\subseteq E_k$ be an embedding into an injective envelope, where $A/A_{-1}$ is defined to be $A$.
Let $f_k:A\to E_k$ be the corresponding morphism with kernel $A_{k-1}$. 

Set $I_n=\bigoplus_{k\leq n} E_k$, and
let  $\phi_n:A_n\to I_n$ be the morphism

\centerline{
$\phi_n=\sum_{k=0}^n\, f_k\vert_{A_n}$.} 

The  commutative diagram

\hskip3cm $A_0\longrightarrow A_1\longrightarrow A_2\dots$
 
\hskip2.7cm  $\phi_0\downarrow \hskip4mm\phi_1\downarrow
\hskip4mm\phi_2 \downarrow$

\hskip3.1cm $I_0\hskip1mm\longrightarrow
 I_1 \hskip1mm\longrightarrow I_2\dots$
 
 \noindent
 obviously satisfies conditions (a) and (b). Applying
the Snake Lemma to the commutative diagram

\hskip3cm $0\to A_{n-1}\longrightarrow A_n\longrightarrow A_n/A_{n-1}\to 0$
 
\hskip3.1cm  $\phi_{n-1}\downarrow \hskip7mm\phi_n\downarrow
\hskip9mm f_n \downarrow$

\hskip3cm $0\to I_{n-1}\hskip1mm\longrightarrow
 I_n \hskip1mm\longrightarrow\hskip3mm E_n\hskip3mm\to 0,$

\noindent we deduce that the natural map 
$I_{n-1}/A_{n-1}\to I_{n}/A_{n}$ is injective, which proves that (c) holds.

\end{proof}

Let  $F:{\mathcal C}\to {\mathcal B}$
be a left exact functor from an LA+ category $\mathcal{C}$ to an abelian category $\mathcal{B}$ satisfying axioms (LA1) and (LA2).  In what follows, ${\mathcal B}$ will be the category of $\K$-vector spaces or ${\mathcal C}$ itself.

\begin{lemma}\label{limF} Suppose that $F(\bigcup_{n\geq 0}\, A_n)=\bigcup_{n\geq 0}\,F( A_n)$, for any  directed  union  $\bigcup_{n\geq 0}\, A_n$. Then

\centerline{
$R^k F(\bigcup_{n\geq 0}\, A_n)=\varinjlim\,R^k F( A_n)$,}

\noindent for all $k\geq 1$.

\end{lemma}

\noindent
\begin{proof} 
Let 

\hskip3cm $A_0\longrightarrow A_1\longrightarrow A_2\longrightarrow\cdots$
 
\hskip3cm  $\hskip2mm\downarrow \hskip9mm\downarrow
\hskip10mm \downarrow$

\hskip3.1cm $I_0\hskip1mm\longrightarrow I_1
\hskip1mm\longrightarrow I_2\longrightarrow\cdots$

\noindent be the commutative diagram of Lemma \ref{reso}.

Set $A:=\cup_{n\geq 0}\, A_n$ and 
$I:= \cup_{n\geq 0}\,I_n$
Since the objects $I_i$ are injective, the morphisms
$I_i\to I_{i+1}$ split. Thus by Lemma \ref{env},
$I\simeq\oplus_{n\geq 0} I_{n}/I_{n-1}$ 
is an injective object. Therefore,
$R^1\,F(A)$  is the cokernel of the 
map 

\centerline{
$F(I)\to F(I/A)$.}

However,

\centerline{$F(I)=
\varinjlim F(I_n)$ and
$F(I/A)=
\varinjlim F(I_n/A_n)$.}

\noindent Since $R^1F(A_n)$  is the cokernel of the map 

\centerline{
$F(I_n)\to F(I_n/A_n)$,}

\noindent it follows that 
$R^1F(A)=
\varinjlim R^1F(A_n)$.

Thus the lemma is proved for $k=1$. Since for $k\geq 2$, we have

\centerline{$R^kF(A)= R^{k-1}F(I/A)$
and $R^kF(A_n)= R^{k-1}F(I_n/A_n)$,}

\noindent the lemma follows, by induction over $k$.
\end{proof}

\bigskip
 
 As an immediate corollary, we obtain:
 
\begin{lemma}\label{limext}  Let $X\in{\mathcal C}$ be an object of finite length. For any directed union 
$\bigcup_{n\geq 0}\,A_n$, we have

\centerline{
$\Ext^{1}_{\mathcal C}(X, \bigcup_{n\geq 0}\,A_n)
= \varinjlim \Ext^{1}_{\mathcal C}(X, A_n)$.}
\end{lemma}

\section{The category of extended $\nabla(n)$-objects}\label{st(n)-cat}\label{section7}

\noindent 
Let ${\mathcal C}$ be an LA+ category as before.
In this section, we fix a nonnegative integer $n$ once and for all.

A filtration ${\mathcal G}_*$ of an object $M\in{\mathcal C}$ is called a {\it standard filtration of type} $\nabla(n)$ if,  for all $k\geq 0$, 
$\overline{\mathcal G}_k M$ is a direct sum of
copies of the standard object $\nabla(n)$, where the possible number of copies ranges from $0$ to
$\aleph_0$. An object $M\in{\mathcal C}$ is called an 
{\it extended $\nabla(n)$-object} if it  admits a standard filtration of type $\nabla(n)$.

Since the axioms of highest weight categories imply that 
extended $\nabla(n)$-objects are simply direct sums of copies of $\nabla(n)$, the notion of an extended $\nabla(n)$-object was not considered in \cite{cps88}.  This does not hold for LA+ categories in general, so we will need to study the structure of extended $\nabla(n)$-objects.

In the section, we show that the quotient $M/N$ of two
extended $\nabla(n)$-objects is  an extended $\nabla(n)$-object. We also show that successive extensions of extended $\nabla(n)$-objects are again  extended $\nabla(n)$-objects.

\subsection{Notation for Section \ref{st(n)-cat}}

 Let ${\mathcal S}^0(n)$ be the category
whose objects are direct sums of copies of $\nabla(n)$, and let
${\mathcal S}^1(n)$ be the category
of extended $\nabla(n)$-objects.

For  technical reasons, we also need the notion of a doubly
extended $\nabla(n)$-object.
A filtration ${\mathcal G}_*$ of an object $M\in{\mathcal C}$ is called an {\it extended standard filtration of type} $\nabla(n)$ if 
$\overline{\mathcal G}_k M$ is an extended $\nabla(n)$-object for all $k\geq 0$. An object $M\in{\mathcal C}$ is called a
{\it doubly extended $\nabla(n)$-object} if it  admits an
extended standard filtration of type $\nabla(n)$.
Let ${\mathcal S}^2(n)$ be the category of 
doubly extended $\nabla(n)$-objects. 

Fortunately we do not need to go further because it will be
proved that ${\mathcal S}^2(n)={\mathcal S}^1(n)$.

\subsection{A basic lemma}

\begin{lemma}\label{basic} 
Let $M$ be a doubly extended $\nabla(n)$-object.

\begin{enumerate} 

\item[(a)] We have

\centerline{$\Hom_{\mathcal C}(\nabla(n)/L(n),M)=0$ and 
$\Ext^1_{\mathcal C}(\nabla(n)/L(n),M)=0$.}

\item[(b)] The restriction map

\centerline{$\Hom_{\mathcal C}(\nabla(n),M)\to 
\Hom_{\mathcal C}(L(n),M)$}

\noindent is an isomorphism.

\item[(c)] The category  ${\mathcal S}^0(n)$ is 
semisimple.

\item[(d)] Let $N$ and $H$ be subobjects of $M$ such that
$N\in {\mathcal S}^2(n)$ and $H\in {\mathcal S}^0(n)$.
Then  $N\cap H$ belongs to ${\mathcal S}^0(n)$.

\end{enumerate}
\end{lemma}

\noindent {\it Proof of (a).} The socle of
a doubly extended $\nabla(n)$-object is a direct sum of copies of
$L(n)$, therefore ${\mathcal F}_{n-1} M=0$. Since
$\nabla(n)/L(n)$ is in ${\mathcal C}(n-1)$, we deduce that

\centerline{$\Hom_{\mathcal C}(\nabla(n)/L(n),M)=0$.}

\smallskip
We now prove that $\Ext^1_{\mathcal C}(\nabla(n)/L(n),\nabla(n))=0$.
By the definition of standard objects and the canonical filtration, we have

\centerline{${\mathcal F}_{n-1}(\nabla(n)/L(n))=\nabla(n)/L(n)$
and ${\mathcal F}_{n-1} (I(n)/\nabla(n))=0$.}

\noindent Therefore, 
$\Hom_{\mathcal C}(\nabla(n)/L(n),I(n)/\nabla(n))=0$, from which it follows that

\centerline{$\Ext^1_{\mathcal C}(\nabla(n)/L(n),\nabla(n))
=0$.}

\noindent 
Since $\nabla(n)/L(n)$ has finite length, the functor
$\Hom_{\mathcal C}(\nabla(n)/L(n),-)$ commutes with
directed unions.
  Using Lemma \ref{limext}, we successively prove that

\centerline{$\Ext^1_{\mathcal C}(\nabla(n)/L(n),M)=0$}

\noindent for 
$M\in{\mathcal S}^0(n)$, then for
$M\in{\mathcal S}^1(n)$ and finally for $M\in{\mathcal S}^2(n)$.

\smallskip
\noindent {\it Proof of (b).}  The kernel of the map

\centerline{$\Hom_{\mathcal C}(\nabla(n),M)\to 
\Hom_{\mathcal C}(L(n),M)$}

\noindent is $\Hom_{\mathcal C}(\nabla(n)/L(n),M)$ and 
its cokernel lies in
$\Ext^1_{\mathcal C}(\nabla(n)/L(n),M)$.
Thus $f$ is an isomorphism by Part (a).

\smallskip
\noindent {\it Proof of (c).} 
Let $M$ and $N$ be objects in ${\mathcal S}^0(n)$. Part (b) implies that

\centerline{$\Hom_{\mathcal C}(M,N)\simeq
\Hom_{\mathcal C}(\Soc M,\Soc N)$,}

\noindent and therefore the category ${\mathcal S}^0(n)$
is semisimple. \qed

\smallskip
\noindent {\it Proof of (d).} There is a unique object
$X\in {\mathcal S}^0(n)$ with 

\centerline{$\Soc X= \Soc(N\cap H)$.}

\noindent  By Part (b), the inclusion maps

\centerline{$\Soc(N\cap H)\to M$,
$\Soc(H\cap N)\to N$, and
$\Soc(H\cap N)\to H$}

\noindent  have unique extensions 

\centerline{$\psi:X\to M$, $\psi':X\to N$ and
$\psi'':X\to H$.}

\noindent
By uniqueness, they are equal, which proves that $X$ embeds  in $N\cap H$.

Since the socle of $X$ and $N\cap H$ are the same, the quotient
$(N\cap H)/X$ is in ${\mathcal C}(n-1)$. However,
$(N\cap H)/X$ lies in $H/X$, which 
is a direct sum of copies of $\nabla(n)$
by Part (c). Hence we have $(N\cap H)/X=0$, i.e.
$N\cap H=X$, which proves Assertion (d).
\qed

\subsection{The canonical standard filtration.}

\begin{lemma} Let $M\in{\mathcal C}$ 
and let $H\subseteq M$ be a subobject in ${\mathcal S}^0(n)$.

\begin{enumerate}
\item[(a)] If $M$ is an extended $\nabla(n)$-object, then 
$M/H$ is an extended $\nabla(n)$-object.

\item[(b)] If $M$ is a doubly extended $\nabla(n)$-object, then 
$M/H$ is a doubly extended $\nabla(n)$-object.
\end{enumerate}

\end{lemma}

\noindent
\begin{proof}  We assume that $M$ belongs to ${\mathcal S}^1(n)$  and prove Assertion (a).  Let 
${\mathcal G}_*$ be a standard filtration of type $\nabla(n)$ on $M$.
It induces a filtration ${\mathcal K}_*$ on $H$.  
Set $M'=M/H$, and let 
 ${\mathcal G}_*'$ be the filtration on $M'$  induced 
 by ${\mathcal G}_*$.

For any $n\geq 0$, there is an exact sequence

\centerline{
$0\to \overline{\mathcal K}_{n}\,H\to
\overline{\mathcal G}_n\,M\to
\overline{\mathcal G}'_n\,M'\to 0$.}

\noindent By Lemma \ref{basic}(b), $\overline{\mathcal K}_{n} H$
belongs to ${\mathcal S}^0(n)$, and by hypothesis, 
$\overline{\mathcal G}_n\,M$
also belongs to ${\mathcal S}^0(n)$. Thus by 
Lemma \ref{basic}(c), $\overline{\mathcal G}'_n\,M'$ is a direct
sum of standard modules, that is, ${\mathcal G}_*'$ is a standard filtration. Hence
$M/H$ is an extended $\nabla(n)$-object.

The proof of Assertion (b) is almost identical. We now assume
that $M$  belongs to ${\mathcal S}^2(n)$, and let
${\mathcal G}_*$ be an extended standard filtration of type $\nabla(n)$ for $M$.
We define induced filtrations ${\mathcal K}_*$ on $H$ and ${\mathcal G}_*'$ on  $M'=M/H$ as before. In the short exact sequence

\centerline{
$0\to \overline{\mathcal K}_{n}\,H\to
\overline{\mathcal G}_n\,M\to
\overline{\mathcal G}'_n\,M'\to 0$,}

\noindent $\overline{\mathcal K}_{n}\,H$ again belongs to
 ${\mathcal S}^0(n)$, and by hypothesis 
$\overline{\mathcal G}_n\,M$
belongs to ${\mathcal S}^1(n)$. Therefore by
Assertion (a), $\overline{\mathcal G}'_n\,M'$ is an 
extended $\nabla(n)$-object, which proves that ${\mathcal G}_*'$ is an
extended standard filtration. Hence
$M/H$ is a doubly extended $\nabla(n)$-object.
\end{proof}

Let $M$ be a doubly extended $\nabla(n)$-object.
By Lemma \ref{basic}(b),
$M$ contains a unique subobject ${\mathcal H}\,M$ in
${\mathcal S}^0(n)$ with $\Soc({\mathcal H}\,M)=\Soc\,M$.
We define a nondecreasing sequence of subobjects

\centerline{
${\mathcal H}_0 M\subseteq {\mathcal H}_1 M\subseteq {\mathcal H}_2 M\dots$}

\noindent as follows.   First set
${\mathcal H}_0 M={\mathcal H} M$. For $m> 0$, define
${\mathcal H}_{m} M$ as the inverse image of 
${\mathcal H}(M/ {\mathcal H}_{m-1} M)$ in $M$.
Thanks to the previous lemma, we verify by induction 
that each quotient $M/ {\mathcal H}_{m-1} M$ is an extended $\nabla(n)$-module. Therefore the procedure is well defined, and by construction, the successive quotients are in ${\mathcal S}^0(n)$.

\begin{lemma}\label{ext-non-ext}  The nondecreasing sequence  
$({\mathcal H}_m M)_{m\geq 0}$ is a  standard filtration for the doubly extended $\nabla(n)$-object $M$.  In particular, any doubly extended $\nabla(n)$-object is an
extended $\nabla(n)$-object.
\end{lemma}

\noindent
\begin{proof} It has just been proved that the successive quotients
are direct sums of copies of $\nabla(n)$.
To prove that the family $({\mathcal H}_m M)_{m\geq 0}$  is a standard filtration, it remains to prove that $M=\bigcup_{m\geq 0}{\mathcal H}_m M$.

For $m\geq 0$, the {\it $m^{th}$-socle} of $M$, denoted by
$\Soc^m\,M$,   is inductively defined by  

\begin{enumerate}

\item[(a)] $\Soc^0\,M=\Soc\,M$  and

\item[(b)] $\Soc^m\,M$ is the inverse image of
$\Soc(M/\Soc^{m-1} M)$ in $M$.

\end{enumerate}

By Axiom (LA3), we have 
$M=\bigcup_{n\geq 0}\,\Soc^m\,M$. Since ${\mathcal H}_m M$ contains
$\Soc^m\,M$, we have $M=\bigcup_{m\geq 0}{\mathcal H}_m M$.

Therefore the  family of subobjects  
$({\mathcal H}_m M)_{m\geq 0}$ exhausts $M$.
\end{proof}

\bigskip

The filtration ${\mathcal H}_*$ is called the {\it  canonical standard filtration of type} $\nabla(n)$ for $M$.

\subsection{Quotients of extended $\nabla(n)$-objects}

\begin{lemma}\label{st1} Let $M,\,N$ be extended $\nabla(n)$-objects with $N\subseteq M$. 

Then the quotient $M/N$ is an extended $\nabla(n)$-object.

\end{lemma}

\noindent
\begin{proof}  It follows from the definition that
${\mathcal H}_0 N=N\cap {\mathcal H}_0 M $. Thus by induction, we have

\centerline{${\mathcal H}_m N=N\cap {\mathcal H}_m M$,}

\noindent for all $m\geq 0$. Let ${\mathcal G}_*$ be the filtration
of $M/N$ induced by the filtration ${\mathcal H}_*$ of
$M$.

There  is a short exact sequence

\centerline{$0\to \overline{\mathcal H}_m N\to
\overline{\mathcal H}_m M\to
\overline{\mathcal G}_m M/N\to 0$.}

\noindent By hypothesis, $\overline{\mathcal H}_m N$ and 
$\overline{\mathcal H}_m M$ are direct sums 
of the standard object $\nabla(n)$. By Lemma \ref{basic}(c),
the category  ${\mathcal S}^0(n)$ is semi-simple. Thus
 $\overline{\mathcal G}_m (M/N)$ is a direct sum of standard objects 
 $\nabla(n)$.
Hence ${\mathcal G}_*$ is a standard filtration of $M/N$, which proves that  $M/N$ is an extended $\nabla(n)$-object.
\end{proof}

\section{Good filtrations  in LA+ categories}\label{section8}

\noindent Let ${\mathcal C}$ be an LA+ category.
A filtration ${\mathcal G}_*$ of an object $M\in{\mathcal C}$ is called {\it good} if,  for all $k\geq 0$, 
$\overline{\mathcal G}_k M$ is an extended $\nabla(n)$-object
for some $n=n(k)$.

To the best of our knowledge, good filtrations originated in
the context of algebraic group representations, in the work of
Humphreys and Jantzen \cite{Hu}. In this setting, 
the canonical filtration first appeared in \cite{Fr}, and our presentation closely follows \cite{M90}.

We will show that the canonical filtration of $M$ is good
whenever $M$ admits a good filtration.
Then we show that a quotient $M/N$   has a good  filtration whenever both objects $M$ and $N$ have  good  filtrations. 
In the context of algebraic groups, the last statement is due to Donkin \cite{D85}.
Our generalization is based on the approach of
 \cite{M00}.

\subsection{The functor $R^1{\mathcal F}_m$}\label{tba}

Let $\mathcal{F}_0M\subseteq \mathcal{F}_1M\subseteq \mathcal{F}_2M\subseteq \cdots$ be the canonical filtration defined in Section \ref{LA+}.  For $m\geq 0$, the endofunctor $\mathcal{F}_m:\mathcal{C}\rightarrow\mathcal{C}$ is left exact by construction. Let $R^1{\mathcal F}_m$ be its first right derived functor.

\begin{lemma}\label{F1} 

For any  $M$ with a good filtration, we have

\centerline{$R^1{\mathcal F}_m M=0$,}

\noindent for all $m\geq 0$.

\end{lemma}

\noindent
\begin{proof}  First we prove that
$R^1{\mathcal F}_m \nabla(n)=0$, or equivalently, that

\centerline{$\phi: {\mathcal F}_m I(n)\to {\mathcal F}_m (I(n)/\nabla(n))$}

\noindent is surjective. Let $X$ be the inverse image in $I(n)$ of
${\mathcal F}_m (I(n)/\nabla(n))$. By definition, there is a short exact sequence

\centerline{$0\to\nabla(n)\to X\to 
{\mathcal F}_m (I(n)/\nabla(n))\to 0$.}

  If $m<n$, then $X/L(n)$ belongs to
${\mathcal C}(n-1)$ and therefore $X\subseteq\nabla(n)$. It follows that ${\mathcal F}_m (I(n)/\nabla(n))=0$, and $\phi$ is obviously surjective.

If $m\geq n$, then ${\mathcal F}_m X=X$, so $X\subseteq\mathcal{F}_mI(n)$ and $\phi$ is again surjective.

Therefore, it is  proved that $R^1{\mathcal F}_m \nabla(n)=0$.
The functor ${\mathcal F}_m$  commutes with
directed unions. Using Lemma \ref{limF}, we succesively prove that
$R^1{\mathcal F}_m M=0$ for $M\in{\mathcal S}^0(n)$, then for
$M\in{\mathcal S}^1(n)$, and finally for any object $M$ with a good filtration.
\end{proof}
 
\subsection{The Zassenhaus condition}\label{Zass}

Let $M$ be an object in ${\mathcal C}$ with a filtration
${\mathcal F}_*'$. Let
${\mathcal G}_*'$ be another filtration of $M$. 
In the next subsection,
${\mathcal F}_*'$ will be the canonical filtration 
${\mathcal F}_*$ and ${\mathcal G}_*'$ will be a good filtration
${\mathcal G}_*$, but for the present subsection
${\mathcal F}_*'$ and ${\mathcal G}_*'$ are arbitrary, but given once and for all.

We say that ${\mathcal G}_*'$ satisfies 
the {\it Zassenhaus condition} if,
for all $k,n\geq 0$,
$${\mathcal G}_{k}' M\cap {\mathcal F}_n' M=
{\mathcal G}_{k-1}' M\cap {\mathcal F}_n' M,\hbox{\  or\ } {\mathcal G}_{k}' M\subseteq {\mathcal G}_{k-1}' M
+{\mathcal F}_n' M.$$
For  $n\geq 0$, let $A(n)$ be the set of integers $k$ such that
$${\mathcal G}_{k}' M\cap {\mathcal F}_n' M\neq
{\mathcal G}_{k-1}' M\cap {\mathcal F}_n' M\hbox{\ and\ }{\mathcal G}_{k}' M\cap {\mathcal F}_{n-1}' M=
{\mathcal G}_{k-1}' M\cap {\mathcal F}_{n-1}' M.$$
Obviously, the Zassenhaus condition is equivalent to 

\centerline{${\mathcal G}_{k}' M\subseteq {\mathcal G}_{k-1}' M
+ {\mathcal F}_n' M$ for all $n\geq 0$ and
$k\in A(n)$.}

\begin{Zlemma}\label{Zlemma}  Assume that the filtration
${\mathcal G}_{*}'$ satisfies the Zas\-sen\-haus condition.

Then $\overline{\mathcal F}_n' M$ has a filtration whose
successive quotients are the objects
$\overline{\mathcal G}_{k}' M$ for $k\in A(n)$.
\end{Zlemma}

\noindent
\begin{proof}  For abelian groups, the proof is obvious. 
In general, the proof is identical, except that it requires the Axiom (LA2).\end{proof}

\subsection{Objects with a good filtration}

\begin{lemma}\label{canonical-good} Let $M$ be an object in ${\mathcal C}$ with a good filtration, and let $(\mathcal{F}_nM)_{n\geq 0}$ be the canonical filtration of $M$. Then

\begin{enumerate}
\item[(a)] $\overline{\mathcal F}_n M$ is an 
extended $\nabla(n)$-object for all $n\geq 0$, and 

\item[(b)] the canonical  filtration of $M$
is good.
\end{enumerate}
\end{lemma}

\noindent
\begin{proof} Let ${\mathcal G}_{*}$ be a good filtration of
$M$. For  $n\geq 0$, let $A(n)$ be the set of integers $k$ such that
$${\mathcal G}_{k} M\cap {\mathcal F}_n M\neq
{\mathcal G}_{k-1} M\cap {\mathcal F}_n M\hbox{\ and\ }{\mathcal G}_{k} M\cap {\mathcal F}_{n-1} M={\mathcal G}_{k-1} M\cap {\mathcal F}_{n-1} M.$$ 
{\it Step 1.} We claim that

\begin{enumerate}

\item[(1)] for any $n\geq 0$, $A(n)$ is the set of $k$ such that
$\overline{\mathcal G}_{k}M$ is a nonzero 
extended $\nabla(n)$-object and

\item[(2)] ${\mathcal G}_{*}$ satisfies the Zassenhaus condition.

\end{enumerate}

Let $k\in A(n)$.  By definition, $\overline{\mathcal G}_{k} M$
is an extended $\nabla(m)$-object for some integer $m$.

Viewed as a subobject of 
$M/{\mathcal G}_{k-1} M$, 
$\overline{\mathcal G}_{k} M$ has a nontrivial  intersection
with $(\mathcal{G}_{k-1}M+\mathcal{F}_nM)/\mathcal{G}_{k-1}M$.
Therefore we have $m\leq n$ and $\overline{\mathcal{G}}_kM=0$.

By Lemma \ref{F1} we have $R^1{\mathcal F}_{n-1}({\mathcal G}_{k-1} M)=0$. It follows that

\centerline{
${\mathcal F}_{n-1} (M/{\mathcal G}_{k-1} M)=
{\mathcal F}_{n-1} M/({\mathcal G}_{k-1} M\cap {\mathcal F}_{n-1} M)$.}

\noindent Since $\overline{\mathcal G}_{k} M$ does not intersect $(\mathcal{G}_{k-1}M+\mathcal{F}_nM)/\mathcal{G}_{k-1}M ={\mathcal F}_{n-1} M/({\mathcal G}_{k-1} M\cap {\mathcal F}_{n-1} M)$, we conclude that $m>n-1$.  That is, $m=n$, which proves the first claim.

Since $\overline{\mathcal G}_{k} M$ is an extended $\nabla(n)$-object, 
it lies in ${\mathcal F}_{n} (M/{\mathcal G}_{k-1} M)$.
By Lemma \ref{F1}, we have $R^1{\mathcal F}_{n} ({\mathcal G}_{k-1} M)=0$, and therefore

\centerline{
${\mathcal F}_{n} (M/{\mathcal G}_{k-1} M)=
{\mathcal F}_{n} M/({\mathcal G}_{k-1} M\cap {\mathcal F}_{n} M)=(\mathcal{F}_nM+\mathcal{G}_{k-1}M)/\mathcal{G}_{k-1}M$,}

\noindent which proves that

\centerline{${\mathcal G}_{k} M\subseteq {\mathcal G}_{k-1} M
+{\mathcal F}_n M$.}

\smallskip
\noindent {\it Step 2} Since the Zassenhaus condition is satisfied,
Lemma \ref{Zlemma} implies that
$\overline{\mathcal F}_n M$ has a filtration whose
successive quotients are the objects
$\overline{\mathcal G}_{k} M$ for $k\in A(n)$.
We have also proved that for $k\in A(n)$, the object
$\overline{\mathcal G}_{k} M$  is an extended $\nabla(n)$-object. By 
Lemma \ref{ext-non-ext},   any doubly
extended $\nabla(n)$-object is an extended $\nabla(n)$-object. Hence
the canonical filtration of $M$ is good.
\end{proof}

\subsection{Quotients of objects with a good filtration}

We conclude with the main result of 
this section.

\begin{proposition}\label{good} Let $N\subseteq M$ be objects in 
${\mathcal C}$ with a good filtration. Then

\begin{enumerate}
\item[(a)]  $M/N$ has a good filtration, and

\item[(b)] 
${\mathcal F}_n(M/N)={\mathcal F}_n M/{\mathcal F}_n N$,
for all $n\geq 0$.
\end{enumerate}

\end{proposition}

\noindent
\begin{proof} Let $n\geq 0$. Assertion (b) is a consequence of 
Lemma \ref{F1}.  It follows that 
$\overline{\mathcal F}_n (M/N)$ is isomorphic to
$\overline{\mathcal F}_n M/\overline{\mathcal F}_n N$.

By Lemma \ref{canonical-good},  $\overline{\mathcal F}_n M$ and 
$\overline{\mathcal F}_n N$ are extended $\nabla(n)$-objects.
Thus by Lemma \ref{st1}, $\overline{\mathcal F}_n (M/N)$
is an extended $\nabla(n)$-object.  Hence the canonical filtration of
$M/N$ is good, which proves Assertion (a).
\end{proof}

\noindent {\bf Remark:} The proof of Lemma
\ref{good} shows that the canonical filtration of
$M/N$ is induced by the canonical filtration of $M$.
By contrast, when $M\subseteq N$ are extended $\nabla(n)$-objects,
the proof of Lemma \ref{st1} shows that the canonical standard filtration of
$M$ induces a standard filtration of $M/N$, but not necesarily the canonical one.

\section{Definition of highest weight categories\\ and generalized highest weight categories}\label{defGHW}

\noindent
Cline, Parshall, and Scott \cite{cps88} introduced the notion of a {\em highest weight category} (HW category). The most important exemples of HW categories are the category $\mathcal{O}$, the category of
modular representations of Chevalley groups, and the category of perverse sheaves with respect to a stratification by cells. 

These categories enjoy many   cohomology vanishing results.  For the applications we have in mind, we will need to work in a somewhat more general setting.  We relax one axiom of a HW category to define the notion of a {\em generalized highest weight category} (generalized HW category), and prove a useful cohomology vanishing result in this setting.


\subsection{Definition of HW categories and generalized HW categories}\label{definitionofHWcat}

Let ${\mathcal C}$ be a LA+ category, with the notation given in Section \ref{LA+}. The category ${\mathcal C}$ is called a {\it HW category} if it satisfies the following two axioms:

\begin{enumerate}

\item[(HW1)]\hskip1cm the canonical filtration of $I(n)$ is good, and

\item[(HW2)]\hskip1cm ${\mathcal F}_n I(n)=\nabla(n)$,
for all $n\geq 0$.

\end{enumerate}

When only the first axiom (HW1) is satisfied, the category 
${\mathcal C}$ is called a {\it generalized HW category}.

Assume that ${\mathcal C}$ is  a generalized HW category. 
Since 

\centerline{$\Ext^1_{{\mathcal C}}(\nabla(n),\nabla(n))=
\Hom_{\mathcal C}(\nabla(n),I(n)/\nabla(n))$,}

\noindent we conclude that
the axiom (HW2) is equivalent to

\hskip1cm $\Ext^1_{{\mathcal C}}(\nabla(n),\nabla(n))=0$, for all $n\geq 0$.

Thus, a HW category is a generalized HW category where the extended $\nabla(n)$-objects are direct sums of standard modules, for all $n\geq 0$.

\subsection{Examples of generalized HW categories}

For simplicity, we have assumed that the category
$\mathcal C$ contains countably many nonisomorphic simple objects. There are similar definitions when
$\mathcal C$ contains only finitely many objects.

\begin{enumerate}

\item[(a)] A  basic example of a HW category is $\Mod(UT_n(\K))$, where $UT_n(\K)$ is the associative algebra of $n\times n$ upper triangular matrices \cite{cps88}.

\item[(b)]For a
finite dimensional local $\K$-algebra $A$, let $UT_n(A)=UT_n(\K)\otimes A$. Then 
$\Mod(UT_n(A))$ is a basic example of a generalized HW category.

\item[(c)] With the usual  basis $\{e,h,f\}$ of $\fsl_2(\K)$,
let  ${\mathcal O}'_{0}$ be the category of $\fsl_2(\K)$-modules $M$ such that

\begin{enumerate}
\item[(i)] The Casimir operator $h^2/2 +ef+fe$ acts locally nilpotently,

\item[(ii)] $M =\bigoplus_{n\leq 0} M^{(n)}$, where
$M^{(n)}$ is the generalized eigenspace

\centerline{$M^{(n)}=\{m\in M\mid (h-n)^N.m=0$ for $N\gg 0\}$.}
\end{enumerate}

\noindent The category ${\mathcal O}'$ has two simple objects, the Verma module $S(0):=M(-2)$ of highest weight $-2$, and the trivial module $S(1)=\K$.  Relative to this order, 
${\mathcal O}'$  is a generalized HW category. We have $\nabla(0)=M(-2)$
and $\nabla(1)$ is the ${\mathcal O}$-dual of the Verma module $M(0)$.

This example easily generalizes to any simple Lie algebra.
\end{enumerate}

\subsection{About the definition of HW and generalized HW categories}

Our definition of HW categories is more constructive and essentially equivalent to the original one, namely Definition 3.1 of \cite{cps88}.  First, it is easy to see that the postulated object $A(n)$
in \cite{cps88} is necessarily the standard module $\nabla(n)$.
Next,  by Lemma \ref{canonical-good}(b), Axiom (HW1) is equivalent to
\begin{enumerate}
\item[{\rm (HW$1^\prime$)}] $I(n)/\nabla(n)$ has a good filtration.
\end{enumerate}
\noindent Therefore  ${\mathcal C}$ is  a generalized HW category
if and only if it satisfies the following axiom:
\begin{enumerate}
\item[{\rm (HW)}] For all $n\geq 0$, $I(n)/\nabla(n)$ has a  filtration
${\mathcal G}_*$  such that, for all $k\geq 0$,
$\overline{\mathcal G}_k (I(n)/\nabla(n))$
is an extended $\nabla(m_k)$-object for some $m_k>n$.
\end{enumerate}

However, in our setting, it would be unnatural to require  Axiom 3.1(c)(iii) of \cite{cps88}, namely that the multiplicities $[I(n):\nabla(m)]$ are always finite. 
That's why our definition of a good filtration is a bit different.  While Cline, Parshall, and Scott \cite{cps88} required that successive quotients be standard objects $\nabla(n)$, we require only that these quotients be extended $\nabla(n)$-objects.  Moreover, our index set $\Z_{\geq 0}$ has a minimal element.
Thus Axiom 3.1(b) of \cite{cps88} means that standard objects have finite length, i.e., this axiom is equivalent to our Axiom (AX3).

\section{ Ext-vanishing in generalized HW categories}

\noindent
In Sections \ref{section7} and \ref{section8}, we have exhausted all elementary properties of good filtrations in LA+ categories.  We now use these properties to prove an Ext-vanishing result for generalized HW categories.  For clarity, we have also included the corresponding statement for HW categories.

From now on, let ${\mathcal C}$ be a generalized HW category.

\subsection{Height of objects in the category ${\mathcal C}$}

By definition, the {\it height} $\Ht\, X$ of an object $X$ is the smallest integer $n$ such that

\centerline{$\mathcal F_n X\neq 0$.}

\noindent Equivalently, it is the smallest 
integer $n$ such that $X$ contains a copy of $L(n)$.
Therefore, $\Ht\, X=\Ht(\Soc\,X)$.

\begin{lemma}\label{ht1} Let $X$ be an object with a
good filtration,  and let $I\supseteq X$ be an injective envelope of $X$.

\begin{enumerate}
\item[(a)] The quotient  $I/X$ has a  good filtration.

\item[(b)] We have $\Ht\,I= \Ht\,X$ and  $\Ht\,I/X\geq \Ht\,X$.

\item[(c)] Moreover, if ${\mathcal C}$ is a HW category, then
$\Ht\,I/X> \Ht\,X$.
\end{enumerate}

\end{lemma}

\noindent
\begin{proof} By Axiom (HW1), $I$ has a good 
 filtration.  Therefore, $I/X$ has a  good filtration by 
 Proposition \ref{good}(a), which proves Assertion (a).
 
  Since $X$ and $I$ share the same socle, we have
 $\Ht\,I= \Ht\,X$. Set $n=\Ht\,X$.  By Proposition \ref{good}(b),
  we have

\centerline{
${\mathcal F}_{n-1}(I/X)=
{\mathcal F}_{n-1} I/{\mathcal F}_{n-1} X$,}

\noindent hence ${\mathcal F}_{n-1}(I/X)=0$ and 
$\Ht(I/X)\geq n=\Ht(X)$, which proves Assertion (b).

Assume now that ${\mathcal C}$ is a HW category.
The socle of $X$ can be decomposed as

\centerline{$\Soc\,X=L(n)^m\oplus Y$,}

\noindent where $\Ht\,Y>n$, and $m$ is a positive integer or $\aleph_0$. By Proposition \ref{good}(a), there is a short exact sequence

\centerline{
$0\to {\mathcal F}_{n} X \to {\mathcal F}_{n}I
\to  {\mathcal F}_{n}(I/X)\to 0$.}

Since $I(n)^m$ is the injective envelope of
${\mathcal F}_{n} \Soc X=L(n)^m$, we have
${\mathcal F}_{n} I={\mathcal F}_{n} I(n)^m=
\nabla(n)^m$. By Lemma \ref{basic}(b),
the subobject $\nabla(n)^m$ lies in $X$, and therefore

\centerline{${\mathcal F}_{n}(I/X)=0$,}

\noindent i.e. $\Ht(I/X)>n=\Ht\,X$.
\end{proof}

\subsection{ Minimal morphisms and minimal resolutions}

A morphism $f:X\to Y$ is called
{\it minimal} if 

\centerline{$\Soc X\subseteq \Ker\,f$ and
$\Soc Y\subseteq \Image\,f$.}

\noindent A resolution 

\centerline{$0\to X\buildrel d_0\over\rightarrow 
I_0\buildrel d_1\over\rightarrow I_1\buildrel d_2\over\rightarrow I_2\buildrel d_3\over\rightarrow\cdots$}

\noindent
of an object $X$ is called {\it minimal} if 
the morphisms $d_i$ are minimal for all $i>0$. Equivalently, it means that the differential of the complex 

\centerline{$\Soc\,I_0\to \Soc\,I_1\to \Soc\,I_2\to\cdots$}

\noindent  is zero. 
It follows from Lemma \ref{env} that each object $M\in{\mathcal C}$ admits a minimal injective resolution: indeed, we can take an injective envelope $I_0$ 
of $M$ and then build the resolution inductively, letting $I_n$ be an injective envelope of $\Coker\,d_{n-1}$ for all $n>0$.

\begin{lemma}\label{heigth} Let $M$ be an object with a
good  filtration and let

\centerline{$0\to M \buildrel d_0\over\rightarrow
I_0\buildrel d_1\over\rightarrow I_1\buildrel d_2\over\rightarrow I_2\buildrel d_3\over\rightarrow\cdots$}

\noindent be a minimal injective resolution of $M$.  Then

\begin{enumerate}

\item[(a)] $\Ker\,d_i$ has a good filtration, for all $i\geq 1$ and

\item[(b)] $\Ht\,M=\Ht\,I_0\leq \Ht\,I_1
\leq \Ht\,I_2\leq \cdots$.

\item[(c)] Moreover, if ${\mathcal C}$ is an HW category, we have

\centerline{$\Ht\,M=\Ht\,I_0< \Ht\,I_1
< \Ht\,I_2<\dots$.}

\end{enumerate}

\end{lemma}

\noindent
\begin{proof} There are  short exact sequences

\centerline{$0\to\Image d_i\to I_i\to \Image d_{i+1}\to 0$.}

 \noindent Thus the lemma is proved inductively using
 Lemma \ref{ht1}.
\end{proof}

\subsection{ An $Ext$-vanishing theorem}

 We will now state an easy  cohomology vanishing result in generalized HW categories. For clarity, we will also 
state the stronger version obtained by specializing to HW categories.

Let $M$ be an object of a generalized HW category $\mathcal{C}$, and
let $k\geq 0$. We inductively define a sequence
$m(0)\leq m(1)\leq m(2)\leq\cdots$ in $\Z_{\geq 0}\cup\{\infty\}$
as follows. With the convention that $\Min\,\emptyset=\infty$, let

\centerline{$m(0)=\Min 
\{n\in\Z_{\geq 0}\mid
\Ext^{n}_{\mathcal C}(L(0),M)= 0\}$,}

and for $k>0$, let

\centerline{$m(k)=\Min\{n\geq m(k-1)\mid  
\Ext^{n}_{\mathcal C}(L(k),M)=0\}$.}

\begin{thm}\label{generalized HWvanishing} Assume that  $M$ has a good filtration. 

\begin{enumerate}

\item[(a)] If ${\mathcal C}$ is a generalized HW category, then 

\centerline{$\Ext_{\mathcal C}^n(L(k),M)=0$, for any 
$n\geq m(k)$.}

\item[(b)] If  ${\mathcal C}$ is an HW category, 
then $m(0)\leq 1$ and

\centerline{$m(k)\leq 1+m(k-1)$, for all $k>0$.}

\noindent In particular, all the $m(k)$ are finite when $\mathcal{C}$ is an HW category.

\end{enumerate}
\end{thm}

\noindent
\begin{proof} {\it Step 1.}  Let $0\to M\to
I_0\to I_1\to I_2\to\cdots$ be a minimal injective resolution of $M$. We claim that $\Ext_{\mathcal C}^n(L(k),M)
\simeq \Hom_{\mathcal C}(L(k),I_n)$ for all $k,n\geq 0$.

\medskip

Since 
$\Hom_{\mathcal C}(L(k),I_n)=\Hom_{\mathcal C}(L(k),\Soc\,I_n)$,
the induced complex

\centerline{$\Hom_{\mathcal C}(L(k),I_0)\to 
\Hom_{\mathcal C}(L(k),I_1)\to 
\Hom_{\mathcal C}(L(k),I_2)\to\cdots$}

\noindent has zero differentials, which proves the claim.

\medskip
\noindent {\it Step 2.} {\em Now we prove by induction on $k$ that $\Ht(I_n)>k$ whenever $n\geq m(k)$.}

\medskip

For $k=0$, the hypothesis that $\Ext_{\mathcal C}^{m(0)}(L(k),M)=0$ is equivalent to $\Hom_{\mathcal C}(L(0),I_{m(0)})=0$.  Hence we have
$\Ht(I_{m(0)})>0$. By Lemma \ref{heigth}(b), we deduce that $\Ht(I_n)>0$ for any $n\geq m(0)$.

Now let $k>0$ be arbitrary. By the induction hypothesis, we already 
know that $\Ht(I_{m(k)})>k-1$. 
The hypothesis that $\Ext_{\mathcal C}^{m(k)}(L(k),M)=0$ is equivalent to $\Hom_{\mathcal C}(L(k),I_{m(k)})=0$, hence
$\Ht(I_{m(k)})>k$. By Lemma \ref{heigth}(b), we deduce that $\Ht(I_n)>k$ for any $n\geq m(k)$.

\medskip
\noindent {\it Step 3.} {\em It follows that
$\Hom_{\mathcal C}(L(k),I_{n})=0$ for any $n\geq m(k)$,
which proves the Assertion (a).
Assertion (b) is an obvious consequence of Lemma \ref{heigth}(c).}
\end{proof}

\bigskip

We will later use the following corollary of Theorem \ref{generalized HWvanishing}:

\begin{cor}\label{corvanishing} Let $\mathcal C$ be a generalized HW category such  that 

\centerline{$\Ext_{\mathcal C}^1(L(0),L(0))=0$
and $\Ext_{\mathcal C}^2(L(1),L(0))=0$.} 

Then
$\Ext_{\mathcal C}^n(L(0),L(0))=0$ and
$\Ext_{\mathcal C}^{n+1}(L(1),L(0))=0$ 
for all $n\geq 1$.
\end{cor}

\vskip1cm
\centerline{\bf \Large PART C: The category of smooth 
$\widehat{\fsl}_2(J)$-modules}
\vskip1cm

\noindent An {\it augmented} Jordan algebra is a unital Jordan
algebra $J$ endowed with an  ideal $J_+$ 
of codimension $1$, given once and for all.

\section{Proof of Theorems \ref{finitegen} and \ref{finitepres}}

\subsection{Nil Jordan algebras}

A nonunital Jordan algebra $I$ is called {\it nil 
of  index $n$} if $x^n=0$ for all $x\in I$. Recall the following result, due to Zelmanov \cite{Zelmanov79}.

\begin{Zthm}\label{ZNil}  Any   finitely generated nil Jordan
algebra of  index $n$ is nilpotent.
\end{Zthm}

Let $I$ be a nonunital Jordan algebra.
 For $n\geq 1$, let $I^{(n)}$ be the linear span of all $a^n$ with $a\in I$. The space $I^{(n)}$ is different from the ideal $I^n$, which is the span of all products of $n$ elements in $I$.
The following lemma is well-known

\begin{lemma}\label{ideal} The space
$I^{(n)}$ is an ideal.
\end{lemma}

\begin{proof}
For any integers $p,\,q\geq 0$, let $S_{p,q}$ be the Jordan 
polynomials in two generators defined by the identity

\centerline{
$(x+ty)^n=\sum_{p+q=n}\,t^q S_{p,q}(x,y)$.}

Since the free Jordan algebra on two generators is special, it is enough to check the following two Jordan identities in the free associative algebra on the generators $x$ and $y$, where $a.b$ denotes the symmetrized product $\frac12(ab-ba)$ of elements $a,b$ in the free associative algebra.

\begin{enumerate}

\item[(a)] $y.x^n=S_{1,n}(y,x)-S_{1,n-1}(x.y,x)$.

\item[(b)] $n x^{n+1}=S_{1,n-1}(x^2,x)$,

\end{enumerate}

These identities imply that

\hskip1cm $y.I^{(n)}\subset I^{(n)}+I^{(n+1)}$ for any $y\in I$ and

\hskip1cm $I^{(n+1)}\subset I^{(n)}$, 

\noindent which proves the lemma.

\end{proof}

\begin{Zcor}\label{Zcor} Let $I$ be a finitely generated nonunital Jordan algebra and let 
$n\geq 1$.  Then

\begin{enumerate}
\item[(a)] the space
 $I/I^{(n)}$ is finite dimensional and
 
 \item[(b)] $I^{(n)}$ contains $I^N$ for some $N\gg 0$.

\end{enumerate} 
\end{Zcor}

\noindent
\begin{proof} By Lemma \ref{ideal},
$I/I^{(n)}$ is a nil Jordan algebra of index $n$, so by  Zelmanov's Theorem \ref{ZNil},   $I/I^{(n)}$ is a nilpotent Jordan algebra.  Hence $I/I^{(n)}$ is finite dimensional and $I^{(n)}$ contains $I^N$ for
some $N$.
\end{proof}

\subsection{Proof of Theorem \ref{finitegen}}
\label{fgproof}

\begin{proof}  It is obviously enough to prove the statement for
the free Jordan algebra $J(D)$ on $D$ generators $x_1,\dots, x_D$. 
This Jordan algebra has a natural grading 

\centerline{$J(D)=\bigoplus_{k\geq 0} J_k(D),$}

\noindent where $J_0(D)=\K$ and $J_1(D)$ is the linear span
of $x_1,\dots,x_D$. Set $J_+:=\bigoplus_{k>0} J_k(D)$ and ${\mathcal A}:={\mathcal U}_n(J(D))$.  
The grading on $J$ induces an obvious grading

\centerline{${\mathcal A}=\oplus_{k\geq 0}\,{\mathcal A}_k$,}

\noindent and we write ${\mathcal A}_+$ for $\bigoplus_{k> 0}\,{\mathcal A}_k$.  Since $\rho_n(1)=n$, we have ${\mathcal A}_0=\K$
and the nonunital algebra ${\mathcal A}_+$ is generated by $\rho_n(J_+(D))$. The defining relation

 \centerline{$\displaystyle{\sum_{\sigma \vdash\ n+1} \sgn(\gs)
|C_\gs|\, h_\gs(a)}$}

\noindent implies that $\rho_{n+1}(J_+^{(n+1)})$ lies in 
${\mathcal A}_+^2$. By Zelmanov's Nil Corollary \ref{Zcor}, 
the space $J_+/J_+^{(n+1)}$ is finite dimensional. Since the
induced map

\centerline{$\rho_{n+1}:J_+/J_+^{(n+1)}\to {\mathcal A}_+/{\mathcal A}_+^2$}

\noindent is surjective, ${\mathcal A}_+/{\mathcal A}_+^2$ is finite dimensional. Hence ${\mathcal A}$ is finitely generated.
\end{proof}

\subsection{A finiteness lemma}

Let $J$ be an augmented Jordan algebra. Since

\centerline{$\widehat\fsl_2(J)\simeq\fsl_2\ltimes
\widehat\fsl_2(J_+)$,}

\noindent the  spaces
$\bigwedge^k \widehat\fsl_2(J_+)$ and
$H_k(\widehat\fsl_2(J_+))$
are $\fsl_2$-modules and their possible eigenvalues 
$\ell$ are even  integers in $[-n,n]$. Let

\centerline{$(\bigwedge^k \widehat\fsl_2(J_+))_\ell$ and $H_k(\widehat\fsl_2(J_+))_\ell$}

\noindent be their respective  eigenspaces of eigenvalue $\ell$.

\begin{lemma}\label{finitehomology} Assume that $J$ is finitely generated.

Then  $H_k(\widehat\fsl_2(J_+))_{2k}$ is finite dimensional.
\end{lemma}

\begin{proof} For any elements 
$a_1,\dots,a_k\in J_+$, the chain
$e(a_1)\wedge\dots\wedge e(a_k)$ is a 
cycle. We write
$[a_1\wedge\dots\wedge a_k]$ for the image of
$e(a_1)\wedge\dots\wedge e(a_k)$ in
$H_k(\widehat\fsl_2(J_+))_{2k}$. Since
$(\bigwedge^k \widehat\fsl_2(J_+))_{2k}$ is the linear span of the cycles $e(a_1)\wedge\dots\wedge e(a_k)$, the 
space $H_k(\widehat\fsl_2(J_+))_{2k}$ is the linear span of the symbols $[a_1\wedge\dots\wedge a_k]$.

We have

\centerline{$d h(a)\wedge e(x)\wedge\omega
  =2e(xa)\wedge \omega+e(x)\wedge \ad(h(a))(\omega)$,}

\noindent
for all $\omega\in\wedge^{k-1}J_+$.

We consider $J$ and $\wedge^{k-1} J$ as $J$-spaces.
The previous formula is equivalent to

\centerline{
$[\rho(a)(x)\wedge\tau +x\wedge\rho(a)(\tau)]=0$,}

\noindent for any $x,a\in J$ and 
$\tau \in \bigwedge^{k-1} J_+$.
Let $\sigma=(\sigma_1,\dots,\sigma_m)$ be a partition of 
$2k-1$ into $m$ parts. Since $\epsilon(\sigma)=-(-1)^m$,
we have

\centerline{
$\epsilon(\sigma)[x\wedge\rho_\sigma(a)(\tau)]
=-[\rho_\sigma(a)(x)\wedge\tau].$}

Note that $\bigwedge^{k-1}J_+$ is a dominant $J$-space of level
$2(k-1)$. By Theorem \ref{dominant}, we have

\centerline{$\displaystyle{\sum_{\gs\ \vdash\ 2k-1} \hbox{sgn}(\gs)|C_\gs|\rho_\gs(a)(\tau)=0}$.}

We deduce that 

\centerline{$\displaystyle{\sum_{\gs\ \vdash\ 2k-1} [\hbox{sgn}(\gs)|C_\gs|\rho_\gs(a)(x)\wedge\tau]=0}$.}

For $x=a$, we deduce that $[a^{2k}\wedge\tau]=0$.
Therefore the map 

\centerline {$\bigwedge^{k}J_+\to H_k(\widehat\fsl_2(J_+))_{2k},\hskip5mm
a_1\wedge\dots\wedge a_k\mapsto [a_1\wedge\dots\wedge a_k]$}

\noindent factors through $\bigwedge^{k}(J_+/J_+^{(2k)})$.
By Zelmanov's Nil Corollary \ref{Zcor}, 
$\bigwedge^{k}(J_+/J_+^{(2k)})$ is finite dimensional,
so $H_k(\widehat\fsl_2(J_+))_{2k}$ is finite dimensional.
\end{proof}

\subsection{Proof of Theorem \ref{finitepres}}
\label{fpproof}

As for the proof of Theorem \ref{finitegen},
it is enough to prove Theorem \ref{finitepres}
for the  free Jordan algebras $J(D)$.  As a quotient of some $\widehat{\mathfrak{sl}}_2(J(D))$, the Lie algebra $\widehat{\mathfrak{sl}}_2(J)$ coming from a finitely presented Jordan algebra $J$ is then obviously finitely generated.  Its relations are given by the  relations in the finite presentation of $\widehat{\mathfrak{sl}}_2(J(D))$ and the (finitely many) additional relations coming from elements $x\ot r$, where $x$ is a basis element of $\mathfrak{sl}_2$ and $r$ is a relation in the finite presentation of the Jordan algebra $J$.

Given a 
fini\-tely presented Lie algebra $\fu$, it is clear that any Lie algebra $\fg$ containing $\fu$ is
also finitely presented if $\dim \fg/\fu<\infty$.
Thus the following reformulation of Theorem 
\ref{finitepres} is slightly stronger.

\begin{citethmfinitepres} The Lie algebra 
  $\widehat{\fsl}_2(J_+)$ is finitely presented, where $J_+$ is the augmentation ideal $\bigoplus_{k>0} J_k(D)$.
\end{citethmfinitepres}

\begin{proof}
 Since $\widehat{\fsl}_2(J_+)$ is a positively
graded Lie algebra, the statement follows from
the finite dimensionality of
$H_1(\mathfrak{\fsl}_2(J_+))$ and $H_2(\mathfrak{\fsl}_2(J_+))$.

Obviously, we have

$$H_1(\mathfrak{\fsl}_2(J_+))\simeq \bigoplus_{1\leq i\leq D} \fsl_2\otimes x_i,$$

\noindent so $\dim H_1(\mathfrak{\fsl}_2(J_+))=3D<\infty$.

We now prove that $H_2(\widehat{\fsl}_2(J_+))$ is finite dimensional.
By Proposition 2 of \cite{KM}
the $\fsl_2$-module $H_2(\widehat{\fsl}_2(J_+))$ is isotypical of type $L(4)$ and by Lemma \ref{finitehomology}
$\dim H_2(\widehat{\fsl}_2(J_+))_4$ is finite dimensional. Hence

\centerline{$\dim H_2(\widehat{\fsl}_2(J_+))=
5 \dim H_2(\widehat{\fsl}_2(J_+))_4<\infty$.}
\end{proof}

\section{Finite dimensionality of \\the Weyl modules $\Delta(n)$}

\noindent From now on, we assume that the unital Jordan algebra $J$ is finitely generated and 
{\it augmented}, that is, $J$ is endowed with a codimension one ideal 
$J_{+}$ given once and for all.

\subsection{Smooth $\widehat{\fsl}_2(J)$-modules}

A  $\widehat{\fsl}_2(J)$-module $M$ is called {\it smooth} if $\dim\,M\leq \aleph_0$ and
for any $v\in M$ we have

\begin{enumerate}

\item[(a)] $\widehat{\fsl}_2(J_+^n).v=0$ for some $n=n(v)>0$, and

\item[(b)] $U(\widehat{\fsl}_2(J)).v$ is finite dimensional.

\end{enumerate}

\noindent   It is easy to see that 
$$\widehat{\fsl}_2(J_+^n)=(\fsl_2\otimes J_+^n)\oplus\{J_+^n,J_+\},$$
where $\{J_+^n,J_+\}$ is the subspace of $\{J_+,J_+\}$ spanned by the elements $\{a,b\}$ with $a\in J_+^n$ and $b\in J_+$.
Assertion (a) is thus equivalent to the conditions:

\begin{enumerate}

\item[(a1)] $(\fsl_2\otimes J_+^n).v=0$, and

\item[(a2)] $\{J_+^n,J_+\}.v=0$,

\end{enumerate}
for some $n=n(v)>0$.  
Obviously,  smooth $\widehat{\fsl}_2(J)$-modules are  always $(\fG,\SL_2)$-modules.

\subsection{Finite dimensionality of  the Weyl modules $\Delta(n)$}

Let $\K_n$ be the $1$-dimensional $\fG_0$-module on which
$h(1)$ acts by multiplication by $n$ and $h(J_+)$ acts trivially.
Set $\Delta(n)=\Delta(\K_n)$. Obviously, there is a surjective map
$\Delta(n)\to L(n)$.

\begin{thm}\label{finitedim} The $\widehat{\fsl}_2(J)$-module  $\Delta(n)$ is smooth and finite dimensional.
\end{thm}

\noindent
\begin{proof}  Let $v\in \K_n$ be a generator of $\Delta(n)$. 

\smallskip\noindent
{\it Step 1.} First we prove that $f(J_+^{(n)}).v=0$.

Let $a\in J_+$. Since 
$\Delta_{-n}(n)=\K.f^n v$, we have $f(a)^{n}.v=0$,
and therefore

\centerline{$\pi(\,e^{n-1} f(a)^{n}).v=0$.}

\noindent By Lemma \ref{garland},
$\pi(\,e^{n-1} f(a)^{n})$ is, up to a scalar multiple, the coefficient of 
$t^{n-1}$ in the formal series

\centerline{$\left
(\sum_{r=0}^\infty f(a^{r+1})t^r\right)\exp\left(-\sum_{s=1}^\infty \frac{h(a^s)}{s}t^s\right).$}

\noindent Since $a$ belongs to $J_+$,  we have $h(a^s).v=0$, for any $s>0$. Thus 

\centerline{$\,e^{n-1} f(a)^{n}.v=
r f(a^{n}).v$, for some nonzero $r$.} 

\noindent Hence $f(J_+^{(n)}).v=0$. 

\smallskip\noindent
{\it Step 2.} By Zelmanov's Nil Corollary \ref{Zcor}
 $J_+^{(n)}$ contains $J_+^N$ for some $N$.
We now prove that $\widehat{\fsl}_2(J_+^N).\Delta(n)=0$. 

Since $\widehat{\fsl}_2(J_+^N)$ is an ideal and $v$ generates $\Delta(n)$, it is enough to prove that
$\widehat{\fsl}_2(J_+^N).v=0$.

Since
$h(J_+).v$, $e(J_+).v$ and $f(J_+^N).v$ are all zero,
we have 

\centerline{$(\fsl_2\otimes J_+^N).v=0$.}

\noindent Moreover we have 

\centerline{$[e(J_+),f(J_+^N)]=\{J_+^N,J_+\}$ modulo 
$\fsl_2\otimes J_+^N$.} 

\noindent Thus, we also have
$\{J_+^N,J_+\}.v=0$, which proves that

\centerline{$\widehat{\fsl}_2(J_+^N).v=0$.}

\smallskip\noindent
{\it Step 3.}  It remains to prove that $\Delta(n)$ is finite dimensional.
Since $\Delta(n)$ is a $(\fsl_2(J/J_+^N)$-module, we can assume
that $J_+^N=0$. Therefore $J$ is finite dimensional, and
$\Delta(n)$ is finite dimensional by 
Lemma \ref{U_nbasic}(a).

\noindent This completes the proof.

\end{proof}

\bigskip 
\noindent 
{\bf Remark:} A $J$-space $M$ of level $k$ is {\em smooth} if for every $v\in M$, we have $h(J_+^n)v=0$ for $n=n(v)\gg 0$. It follows from Theorem \ref{finitedim} that a dominant smooth $J$-space $M$ is finite dimensional if and only if $\Delta(M)$ is finite dimensional.

\section{Lie algebra cohomology and smooth representations}

\noindent As in the previous Section, let $J$ be a finitely generated augmented Jordan algebra. In this section, we  compare $\hbox{Ext}^*$ in the category
of smooth modules with relative Lie algebra cohomology.

From now on, we consider only representations of finite or countable dimension for Lie algebras and algebraic groups.

\subsection{Generalities about algebraic groups}\label{10.1}

Let $\fg$ be a finite dimensional perfect Lie algebra.
Then $\fg\simeq\fs\ltimes \fu$ where $\fs$ is semi-simple and
$\fu$ is the nilpotent radical.  
Let $G$ be the unique connected and simply connected algebraic group with Lie algebra $\fg$. There is a Levi decomposition $G=S\ltimes U$, where $S$ is semi-simple and $U$ is unipotent, with Lie algebras $\fs$ and $\fu$, respectively.

In what follows, a rational $G$-module will simply be called a 
{\it $G$-module}. Let $\Mod(G)$ be the category of 
$G$-modules.  
It is well known that $\fu$ acts nilpotently 
on any  finite dimensional $\fg$-module.
It follows that any locally finite dimensional $\fg$-module is a  $G$-module, and conversely. 

The group $G$ acts on $U$ as follows: $U$ acts by left multiplication and $S$ by conjugation. With respect to this action,
the ring $\K[U]$ of rational functions on $U$ is a  $G$-module.

Since $\K[U]$ is an injective $U$-module and since $S$ acts
reductively, the $G$-module $\K[U]$ is injective. In fact,
$\K[U]$ is the injective envelope of the trivial module.
Hence $M\otimes \K[U]$ is an injective envelope of $M$, for any simple $G$-module $M$.

\subsection{Generalities about  group cohomology}

Let $\fg$, $\mathfrak{s}$, $\mathfrak{u}$, $G$, $S$, and $U$ be defined as in the previous subsection.
Let $M\in \Mod(G)$. By definition, the {\it relative cohomology group} $H^*(\fg,\fs,M)$ is the cohomology of the Chevalley-Eilenberg cocomplex $\Hom_{\fs}(\bigwedge^p \fg/\fs,M)$.  See
\cite{Ko} or \cite{kumar} for details. Since $\fs$ acts reductively, we have

\centerline{$H^*(\fg,\fs,M)\simeq H^*(\fu,M)^\fs$.}

As an algebraic variety, $U$ is isomorphic to its Lie algebra
$\fu$, thus
the de Rham complex $\bigwedge \fu^*\otimes \K[U]$ is exact.
For any $U$-module $M$, it follows easily that

\centerline{$H^*(\fu, M)\simeq H^*(U, M)$,}

\noindent a result often attributed to Malcev.
As $\fs$ is reductive,  Hochschild proved that 
the relative cohomology $H^*(\fg,\fs,-)$ describes extensions in the relative category 
of $(\fg,\fs)$-modules \cite{H}.\footnote{A more recent account appears in \cite{Mu}.}
Combining this result with the Malcev isomorphism, we deduce the
following well known result:

\begin{lemma}\label{folk} Let $M$ and $N$ be $G$-modules, with $M$ finite dimensional. Then

\centerline{$H^*(\fg,\fs,\Hom_{\K}(M,N))=\Ext^*_{\Mod(G)}(M,N)$.}
\end{lemma}

\noindent

\subsection{The pro-algebraic group $\bSL_2(J)$}

For any integer $n\ge 1$, the Lie algebra
$\widehat{\fsl}_2(J/J_+^n)=\fsl_2\ltimes\widehat{\fsl}_2(J_+/J_+^n)$
is perfect and finite dimensional.
Let $U_n$ be the unipotent group with Lie algebra
$\widehat{\fsl}_2(J_+/J_+^n)$ and set
$G_n=\SL_2\ltimes U_n$. By definition,
$\bSL_2(J)$ is the proalgebraic group $\varprojlim G_n$, and an 
{\it  $\bSL_2(J)$-module} is a direct union
$\bigcup_{n\geq 1}\, M_n$ of  $G_n$-modules $M_n$, where,
as always, we assume that $\dim M\leq \aleph_0$. Since the Lie algebra $\widehat{\fsl}_2(J/J_+^n)$ is perfect and finite dimensional, 
any smooth $\widehat{\fsl}_2(J)$-module is an  $\bSL_2(J)$-module and conversely.

\subsection{The category $\Mod(\bSL_2(J))={\mathcal C}_{sm}(J)$}\label{ordering}

Let $\Mod(\bSL_2(J))$ be the category of 
$\bSL_2(J)$-modules.  By the material in Section \ref{10.1}, $\Mod(\bSL_2(J))$ coincides with the category ${\mathcal C}_{sm}(J)$ 
of smooth $\widehat{\fsl}_2(J)$-modules. Depending on the context,
we will use one  terminology or the other.

Let $n\geq 0$. In what follows, the  simple ${\fsl}_2$-module
$L(n)$  of
dimension $n+1$ will be viewed as a
$\widehat{\fsl}_2(J)$-module with trivial action of $\widehat{\fsl}_2(J_+)$.
It is clear that the family $(L(n))_{n\in\Z_{\geq 0}}$ is a 
complete list of  simple objects in $\Mod(\bSL_2(J))$.

Set $U=\varprojlim U_m$. By definition,  its ring of rational functions is $\K[U]=\varinjlim \K[U_m]$.  Viewing $\K[U]$ as a $\bSL_2(J)$-module, let

\centerline{$I(n)=L(n)\otimes \K[U]$.}

\begin{lemma} Let $n\geq 0$. Then

\begin{enumerate}

\item[(a)] $I(n)$ is the injective envelope of $L(n)$ and

\item[(b)] $\Delta(n)^*$ is the standard object 
$\nabla(n)$ with socle $L(n)$.

\end{enumerate}

\end{lemma}

\noindent
\begin{proof}  Clearly $L(n)\otimes \K[U_m]$ is the injective envelope of $L(n)$ in the category $\Mod(\bSL_2(J/J_+^m)$, which proves Assertion (a), and shows the existence of the standard object $\nabla(n)$.  By construction, the cosocle of $\Delta(n)$ is $L(n)$, so $Soc(\Delta(n)^*)=L(n)$ and $\Delta(n)^*\subseteq I(n)$, from which it follows that $\Delta(n)^*$ is contained in the standard object $\nabla(n)$ with socle $L(n)$.

  Conversely, the $\Delta(n)$ surjects onto the dual
$\nabla(n)^*$ of the standard module, since the highest eigenspace of $\nabla(n)$ is the $J$-space $\K_n$.  Dualizing, we see that $\nabla(n)$ is contained in $\Delta(n)^*$. 
  
  Hence $\Delta(n)^*=\nabla(n)$.  \end{proof}

\bigskip
The definition of $LA^+$-categories is provided in the
PART B,
Section \ref{defLA+}.
Since the standard module $\nabla(n)=\Delta(n)^*$ is finite dimensional by Theorem \ref{finitedim}, we have:

\begin{cor}\label{corLA+} The category $\Mod(\bSL_2(J))$ is a LA+ category.
\end{cor}

\subsection{Smooth cohomology} 

 We will now see how to compute $\Ext$-groups
in the category $\Mod(\bSL_2(J))=\mathcal{C}_{sm}(J)$ of smooth $\widehat{\fsl}_2(J)$-modules.

Let $M$ be a 
smooth $\widehat{\fsl}_2(J)$-module. A relative
$k$-cochain

\centerline{$\omega\in \Hom_{\fsl_2}(\bigwedge^k (\widehat{\fsl}_2(J)/\fsl_2), M)$}

\noindent 
is called {\it smooth} if $k=0$, or 

\centerline{$\omega(x_1\wedge x_2\wedge\cdots\wedge x_k)=0$,
whenever $x_1\in  \widehat{\fsl}_2(J_+^N)$,}

\noindent for $N\gg 0$. Let $C^k_{sm}(M)$ be the space of smooth relative $k$-cochains with values in $M$. By definition, the smooth cohomology

\centerline{$H^*_{sm}(\widehat{\fsl}_2(J),\fsl_2, M)$}

\noindent is the homology of the complex $C_{sm}^*(M)$. It follows that

\centerline{$H^*_{sm}(\widehat{\fsl}_2(J),\fsl_2, M)=
\varinjlim H^*(\widehat{\fsl}_2(J/J_+^n),\fsl_2, M)$.}

\noindent Thus the folklore Lemma \ref{folk} implies the following result:

\begin{lemma}\label{ext=smooth} Let $M,\,N$ be smooth
$\widehat{\fsl}_2(J)$-modules
 with $M$ finite dimensional. Then

\centerline{$\Ext^*_{\Mod(\bSL_2(J))}(M,N)=
H^*_{sm}(\widehat{\fsl}_2(J),\fsl_2,\Hom_\K(M,N))$.}
\end{lemma}

\subsection{Comparison with relative homology.}

 Let $M$ be a finite dimensional $\bSL_2(J)$-module.
In general, it is difficult to compare the smooth cohomology
$H^*_{sm}(\widehat{\fsl}_2(J),\fsl_2,M)$ with the Chevalley-Eilenberg relative homology
$H_*(\widehat{\fsl}_2(J),\fsl_2,M^*)$. 

However, it is possible to do so when $M$ and $J$ have  compatible gradings. A 
{\it positively graded Jordan algebra} is a Jordan algebra endowed with a grading
$J=\bigoplus_{n\geq 0}\, J_n$ such that
$J_0=\K$. It is viewed as an augmented 
Jordan algebra, with augmentation ideal
$J_+:=\bigoplus_{n>0} J_n$. The Lie algebra
$\widehat\fsl_2(J)$ inherits the grading of $J$.

Given a $\Z$-graded vector space 
$E=\bigoplus_{n\in\Z} E_n$, the {\it graded dual} of $E$ is the space
$\bigoplus_{n\in\Z} E_n^*$.

\begin{lemma}\label{compare} Assume that $J$ is a finitely generated positively graded  Jordan 
algebra. Let $M$ be a finite dimensional $\Z$-graded 
$\widehat{\fsl}_2(J)$-module. Then 

\begin{enumerate}
\item[(a)] $M$ is smooth.

\item[(b)] $H^k_{sm}(\widehat{\fsl}_2(J),\fsl_2,M)$ is the graded dual
of $H_k(\widehat{\fsl}_2(J),\fsl_2,M^*)$,
for any integer $k\geq 0$. 
\end{enumerate}
\end{lemma}

\noindent
\begin{proof} Assertion (a) is obvious and Assertion (b) follows  from the fact that 
$C^k_{sm}(M^*)$ is the graded dual of the space 
$H_0(\fsl_2,\bigwedge^k (\widehat{\fsl}_2(J)/\fsl_2)\otimes M)$
of relative chains.
\end{proof}

Assume again that $J$ is a finitely generated positively graded  Jordan 
algebra. The simple modules $L(n)$ can be viewed 
as $\Z$-graded modules, with  grading concentrated in degree $0$. The following statement is
an immediate consequence of Lemmas \ref{ext=smooth} and \ref{compare}.

\begin{cor}\label{Corcompare} For any $k,n\geq 0$,
the group $\Ext^k_{\Mod(\bSL_2(J))}(L(n),L(0))$ is the graded dual of 
$H_k(\widehat{\fsl}_2(J),\fsl_2,L(n))$.
\end{cor}

\subsection{Some Ext-vanishing results.}

\begin{theorem}\label{coh} Let $J$ be the free unital Jordan algebra $J(D)$ on $D>0$ generators.  Then 

\begin{enumerate}

\item[(a)]
$\Ext^k_{\Mod(\bSL_2(J(D))}(L(0),L(0))=0$ for $k=1,\,2,\,3$.

\item[(b)] 
$\Ext^2_{\Mod(\bSL_2(J(D))}(L(2),L(0))=0$.

\end{enumerate}
\end{theorem}

\noindent
\begin{proof}  By Corollary \ref{Corcompare},
$\Ext^k_{\Mod(\bSL_2(J(D))}(L(n),L(0))$ is the graded dual of $H_k(\widehat{\fsl}_2(J(D)),\fsl_2,L(n))$.
Thus Theorem \ref{coh}  is equivalent to the following known results about relative homology:

\begin{enumerate}

\item[(a)] 
$H_1(\widehat{\fsl}_2(J), \fsl_2, \K)=0$, which is obvious,

\item[(b)] 
$H_2(\widehat{\fsl}_2(J), \fsl_2, \K)=0$, due to \cite{AG},

\item[(c)] 
$H_3(\widehat{\fsl}_2(J), \fsl_2, \K)=0$, proved in 
\cite[Prop 3]{KM},

\item[(d)] 
$H_2(\widehat{\fsl}_2(J), \fsl_2, L(2))=0$, proved in 
\cite[Lemma 13]{KM}.
\end{enumerate}
\end{proof}

\section{Is $\Mod(\bPSL_2(J(D)))$ a \\generalized HW category?}

\noindent   Let $J$ be a finitely generated augmented Jordan algebra. Set
$\bPSL_2(J)=\bSL_2(J)/\pm1$.
Then $\Mod(\bPSL_2(J))$ is the category
of $\bSL_2(J)$-modules with even eigenvalues.
The simple modules in 
$\Mod(\bPSL_2(J))$  are ordered by their dimensions and the
$(n+1)^{st}$ simple module is $L(2n)$.
By Corollary \ref{corLA+}, $\Mod(\bPSL_2(J))$ is an LA+ category.

 The free Jordan algebra 
 $J=J(D)$ is  viewed as an augmented Jordan algebra with  augmentation ideal $J_+=\bigoplus_{n>0} J(D)_n$.  In this section, we consider the category $\Mod(\bPSL_2(J(D)))$, relating it to the framework of $\hbox{Ext}$-vanishing and generalized highest weight categories introduced in Part B.


\subsection{Motivation}

Let $J$ be a finitely generated augmented Jordan algebra, and let $p$ be a prime number.
The category $\Mod(\bPSL_2({\overline \F}_p))$ of rational
$\bPSL_2({\overline \F}_p)$-modules and the category $\Mod(\bPSL_2(J))$ look very similar. In fact, Donkin \cite{D82} and Koppinen
\cite{Ko84} showed that $\Mod(\bPSL_2({\overline \F}_p))$ satisfies the axioms of a HW category (though this notion was introduced only some time later in \cite{cps88}).  The category $\Mod(\bPSL_2(J))$ is $LA+$. However,
even in the most favorable case when  
$J$ is the free Jordan algebra $J(D)$, the following proposition shows
that $\Mod(\bPSL_2(J(D)))$ is not a HW category. 

\begin{proposition} For $J=J(D)$,
  $$\Ext^1_{\Mod(\bPSL_2(J))}(\nabla(2),\nabla(2))\simeq
 (J_+/J_+^3)^*.$$
  \end{proposition}

\begin{proof}  The proof follows from  
\cite[Lemma 21]{KM}. The details are left to the reader.
\end{proof}

\smallskip
Nonetheless, for the free Jordan algebra $J(D)$ the Lie algebra
$\widehat{\fsl}_2(J(D))$ is free in a certain category
\cite{KM}, which suggests it should enjoy some cohomology vanishing results. 
This raises the natural question:

\medskip

\centerline{\it Is  the category $\Mod(\bPSL_2(J(D)))$
a generalized HW category?}

\subsection{Some evidence}

Several computations make it seem plausible that $\Mod(\bPSL_2(J(D)))$ might be a generalized HW category.  As these calculations are very technical, we mention only one of them.  If $\Mod(\bPSL_2(J(D)))$ is a generalized HW category, then ${\mathcal F}_n (I(n))$ is an extended $\nabla(n)$-object. This conclusion
is obvious for $n=0$ or $1$, even without assuming that $J$ is free.
The first interesting case is when $n=2$.
For most augmented Jordan algebras, ${\mathcal F}_2 (I(2))$ is not
an extended $\nabla(2)$-object. However, this does hold for the free Jordan algebra $J(D)$.

\begin{lemma} If $J=J(D)$, the $\widehat\fsl_2(J)$-module
${\mathcal F}_2 (I(2))$ is a $\nabla(2)$-object.
\end{lemma}

\noindent
\begin{proof}  The standard module $\nabla(2)$ is the indecomposable module with  short exact sequence
$$0\to L(2)\to\nabla(2)\to \K^D\to 0.$$
The module ${\mathcal F}_2 (I(2))$ is the graded dual 
of the module $M\simeq\Delta_2({\mathcal U}_2(J))$ 
explicitely described in \cite[Lemma 21]{KM}. It follows 
easily that ${\mathcal F}_2 (I(2))$ is a $\nabla(2)$-object.
\end{proof}

\subsection{Consequences}

\begin{thm}\label{GHWhypothesis}  If $\Mod(\bPSL_2(J(D)))$
is a generalized HW category, then the conjectural formula of \cite{KM} for 
$\dim J_n(D)$ would hold.
\end{thm}

\noindent
\begin{proof}  Suppose that $\Mod(\bPSL_2(J(D)))$ is a generalized HW category.
By Theorem \ref{coh}, we have

\centerline{
$\Ext^1_{\Mod(\bPSL_2(J(D)))}(L(0),L(0))=0$, and
$\Ext^2_{\Mod(\bPSL_2(J(D)))}(L(2),L(0))=0$.}

\smallskip\noindent
Note that, relative to the given order, $L(0)$ and $L(2)$ are the first two simple objects in $\Mod(\bPSL_2(J(D))$.  They correspond with the objects denoted $L(0)$ and $L(1)$ in PART B. Thus
by Corollary \ref{corvanishing}, we deduce that 

\centerline{
$\Ext^n_{\Mod(\bPSL_2(J(D)))}(L(0),L(0))=0$, and
$\Ext^{n+1}_{\Mod(\bPSL_2(J(D)))}(L(2),L(2))=0$,}

\smallskip\noindent
for all $n>0$.  By Corollary \ref{Corcompare}, this is equivalent to

\centerline{
$H_n(\widehat{\fsl}_2(J), \fsl_2,\K)=0$ and
$H_{n+1}(\widehat{\fsl}_2(J), \fsl_2, L(2))=0$,}

\smallskip\noindent
for all $n>0$. Note that

\centerline{$H_n(\widehat{\fsl}_2(J), \fsl_2,\K)
\simeq H_n(\widehat{\fsl}_2(J_+))^{\fsl_2}$}

\noindent  and
$H_{n+1}(\widehat{\fsl}_2(J), \fsl_2, L(2))\otimes L(2)$ is the adjoint part of $H_n(\widehat{\fsl}_2(J_+))$.

Thus by
\cite [Corollary 1]{KM}, this implies that
the conjectural  formula for 
$\dim J_n(D)$ holds. The reader should note that
the notation used here conflicts with the notation in \cite{KM}.
The Jordan algebras $J(D)$ and $J_+(D)$ are denoted
$J^u(D)$ and $J(D)$ in \cite{KM}. Moreover, the
Tits-Allison-Gao functor $\widehat\fsl_2$ is denoted
$\fsl_2$ in \cite{KM}.
\end{proof}

However, as noted in the introduction and in Section \ref{FP_infinity}, very recent results posted to arXiv \cite{DH} indicate that the conjectural formula of \cite{KM} does {\it not} hold in general.  As a direct consequence of Theorem \ref{GHWhypothesis}, we see that $\Mod(\bPSL_2(J(D)))$ and the larger category $\mathcal{C}_{sm}(J(D))$ of all smooth $\widehat{\mathfrak{sl}}_2(J(D))$-modules are not generalized highest weight categories:

\begin{cor}\label{notGHWC}  For $D\geq 2$, the categories $\Mod(\bPSL_2(J(D)))$ and $\mathcal{C}_{sm}(J(D))$ are {\it not} generalized highest weight categories.\qed \end{cor}





\end{document}